\documentclass[11pt]{article}
\usepackage[pdftex]{hyperref}
\hypersetup{
  colorlinks   = true, 
  urlcolor     = gray, 
  linkcolor    = red, 
  citecolor   =  blue
}
\usepackage{amsmath,amsfonts,amsthm,amssymb,graphicx,xcolor}
\usepackage[all,2cell,ps]{xy}
\hypersetup{colorlinks=true, citecolor=[rgb]{0,0.5,1}}
\usepackage[margin=1.3 in]{geometry}

\theoremstyle{plain}
\newtheorem{theorem}[equation]{Theorem}
\newtheorem{corollary}[equation]{Corollary}
\newtheorem{lemma}[equation]{Lemma}
\newtheorem{proposition}[equation]{Proposition}

\theoremstyle{definition}
\newtheorem{define}[equation]{Definition}

\newtheorem{remark}[equation]{Remark}
\newtheorem{remarks}[equation]{Remarks}

\newtheorem{question}[equation]{Question}

\newcommand{\p}{\partial}

\newcommand{\IB}{\mathbb{B}}
\newcommand{\IC}{\mathbb{C}}

\newcommand{\IH}{\mathbb{H}}

\newcommand{\IK}{\mathbb{K}}

\newcommand{\IN}{\mathbb{N}}
\newcommand{\IO}{\mathbb{O}}

\newcommand{\IR}{\mathbb{R}}

\newcommand{\IZ}{\mathbb{Z}}

\newcommand{\inj}{\rm inj}

\title{On the Betti Numbers of  Finite Volume Real- and Complex-Hyperbolic Manifolds}  

\author{\small{Luca F. Di Cerbo}\footnote{Partially supported by NSF grant DMS-2104662} \\ \scriptsize{University of Florida}\\ \footnotesize{\textsf{ldicerbo@ufl.edu}} \and \small{Mark Stern}\footnote{Partially supported by Simons Foundation Grant 3553857} \\ \scriptsize{Duke University} \\ \footnotesize{\textsf{stern@math.duke.edu}}}
\date{}

\begin{document}

\maketitle

\begin{abstract}
We obtain strong upper bounds for the Betti numbers of compact complex-hyperbolic manifolds. We use the unitary holonomy to improve the results given by the most direct application of the techniques of \cite{DS17}. We also provide effective upper bounds for Betti numbers of compact quaternionic- and Cayley-hyperbolic manifolds in most degrees.  More importantly, we extend our techniques to  complete finite volume real- and complex-hyperbolic manifolds. In this setting, we develop new monotonicity inequalities  for strongly harmonic forms  on hyperbolic cusps and employ a new peaking argument to estimate $L^2$-cohomology ranks.   Finally, we provide bounds on the de Rham cohomology of such spaces, using a combination of our bounds on $L^2$-cohomology, bounds on the number of cusps in terms of the volume, and the topological interpretation of reduced $L^2$-cohomology on certain rank one locally symmetric spaces. 
\end{abstract}

\vspace{12cm}

\tableofcontents

\vspace{10 cm}

\section{Introduction}\label{introduction}

In theory, the Selberg trace formula can be used to compute Betti numbers of compact locally symmetric spaces. (See, for example, \cite{Wallach} for an introduction to the trace formula.) The expressions produced by the trace formula, however, can be extremely complicated and difficult to understand. DeGeorge and Wallach \cite{DW78} pioneered the use of the trace formula to obtain simpler but coarser bounds on Betti numbers and the multiplicity of other representations arising in $L^2(\Gamma\backslash G)$. In particular, the DeGeorge-Wallach bounds did not require extensive evaluations of orbital integrals. This saving in complexity was achieved (with a concomitant loss of precision) by using test functions in the trace formula that were supported in the lift to $G$ of a  ball about the origin in $G/K$ that can be identified with an {\em embedded} ball in $\Gamma\backslash G/K$.  Their estimates then reduce to estimates on the growth of matrix coefficients associated to the representation under investigation. These techniques were later extended by Savin in \cite{Savin} to non-cocompact arithmetic $\Gamma$. Using related techniques, Xue in \cite{Xue} gave effective estimates on the growth of the first Betti numbers for compact complex-hyperbolic surfaces. Sarnak and Xue \cite{Sarnak} subsequently pushed beyond the embedded ball barrier for surfaces, allowing test functions supported on the lift to $G$ of balls with radius larger than the injectivity radius of $\Gamma\backslash G/K$. They then use estimates on the test function and the size of the $\Gamma$ orbits intersecting these balls to bound rather than evaluate terms arising in the Selberg trace formula that are  associated to non-identity conjugacy classes in $\Gamma$.

In \cite{DS17}, we introduced a method for bounding Betti numbers on  manifolds without conjugate points and with a negative Ricci curvature upper bound. Like the DeGeorge-Wallach technique, our method gives bounds in terms of the volume of the largest embedded geodesic ball in the manifold. The estimates introduced in \cite{DS17}, which we called ``Price inequalities'', play a similar role to that of the estimates on the growth of matrix coefficients in \cite{DW78}, but of course, ours require no locally symmetric hypothesis and easily provide  effective  bounds on Betti numbers. 

Given the recent strong interest in obtaining bounds for cohomology in towers of spaces (see for example \cite{Marshall}, \cite{Bergeron}, \cite{MarShin},  \cite{Lombardi}), both locally symmetric and otherwise, in this paper, we build upon \cite{DS17} to study the Betti numbers of complete  finite volume real-hyperbolic and complex-hyperbolic manifolds. As curvature becomes less pinched (and forms approach middle degree), our estimates become weaker, but K\"ahler holonomy and the associated Hodge theory provides compensating extra structure in the complex hyperbolic case. Unfortunately, we have not yet provided a simple method to harness the special holonomy in the quaternionic and octonionic cases to compensate in all degrees for the reduced pinching. Hence, we focus first on compact complex-hyperbolic manifolds, where we use the K\"ahler property to obtain optimal bounds within the embedded ball barrier. (We  treated compact real-hyperbolic manifolds in \cite{DS17}.) We subsequently sharpen our peaking estimates to extend our techniques to complete finite volume real- and complex-hyperbolic manifolds. Along the way, we also provide effective upper bounds for most (but not all) Betti numbers of compact quaternionic- and Cayley-hyperbolic manifolds.  \\

Before we summarize our results, we introduce a definition.\\

\begin{define}\label{def1} 
Let $(X^n,g)$ be a Riemannian manifold. Let $S\subset X$. Define 
$$V_{min}(S) = \inf_{x\in S} Vol(B_{\inj_x}(x)),$$
where $\inj_x$ is the injectivity radius at $x$, and $B_R(x)$ is the geodesic ball of radius $R$ and center $x$. 
Set $\inj_S:= \inf_{x\in S}  inj_x.$
\end{define}

Throughout the paper, we let  $b^k(M):=\dim_{\IR}H^{k}(M; \IR)$. We can now state our first result.

\begin{theorem}\label{A}
Let $(M^{n}=\Gamma\backslash\textbf{H}^{n}_{\IC}, g_{\IC})$ be a compact complex-hyperbolic manifold, with $-4\leq sec_{g_{\IC}}\leq -1$.   For $k<n$, there exists a positive constant $d(n, k)$ depending only on the dimension and the degree $k$ such that
\[
\frac{b^{2n-k}(M)}{Vol(M)}=\frac{b^{k}(M)}{Vol(M)}\leq d(n, k)V_{min}(M)^{\frac{k-n}{n}}.
\]
\end{theorem}

In \cite{Sarnak}, Sarnak-Xue show that for certain congruence subgroup quotients, we can bound $V_{min}(M)$ below by a power of $Vol(M)$. This gives the following corollary.

\begin{corollary}\label{congruence}
Let $(M^{n}=\Gamma\backslash\textbf{H}^{n}_{\IC}, g_{\IC})$ be a compact complex-hyperbolic manifold, with $-4\leq sec_{g_{\IC}}\leq -1$, where $\Gamma\leq\textrm{PU}(n, 1)$ is a principal congruence subgroup of a cocompact arithmetic subgroup $\Gamma_0$. For every $\epsilon >0$, there exists  $D_{\Gamma_0,\epsilon}>0$ (independent of the level) such that 
\[
\frac{b^{2n-k}(M)}{Vol(M)}=\frac{b^{k}(M)}{Vol(M)}\leq D_{\Gamma_0,\epsilon}d(n, k) Vol(M)^{\frac{2k-2n+\epsilon}{n^2+2n}}.
\]
\end{corollary}

This result improves and generalizes from $H^{k,0}$ to $H^k$ the results in \cite[Theorem 2.3.1]{Yeung1}, which seem to be the best published effective bounds for rank $H^{k,0}$ on compact complex-hyperbolic manifolds of dimension bigger than or equal to three.  We are unaware of any published results using those techniques in higher dimensions.  

For surfaces, our bound on $b^1$  is  sharper  than the one derived by Xue in \cite{Xue}, but weaker than the bound given by Sarnak-Xue \cite{Sarnak} and more recently by Marshall \cite{Marshall}. The bound found by Marshall is a striking $\alpha=\frac{3}{8}$, which  is even  sharp. These bounds both go beyond the embedded ball barrier in their use of the trace formula, and rely on low dimensional arithmetic information.  

In the compact case, our techniques can function as a replacement (or simplification) of estimates of matrix coefficient asymptotic used in \cite{DW78} and \cite{Wallach1}. In order to compare the two techniques, in the appendix we provide a simple method for estimating the relevant matrix coefficients required to estimate the first Betti number  of compact complex-hyperbolic manifolds and show that it reproduces our result. The estimates become somewhat more complex for $k>1$, and are not included in the appendix. We also include the simpler estimation of the matrix coefficients in the real hyperbolic case, and verify that they reproduce the Betti number estimates we previously derived in \cite{DS17}. With two different techniques measuring two very different quantities, it remains a challenge to understand how to combine them to get stronger information than each provides separately.  In the complete finite volume case, the trace formula becomes more complicated, and we are unaware of an extension of the DeGeorge-Wallach argument to this case. The geometric approach, on the other hand, extends to manifolds with rank one cusps to  bound the ranks of their reduced $L^2$-cohomology.

\begin{define}
Let $(M, g)$ be a complete Riemannian manifold. Define the vector space of $L^2$ harmonic $k$-forms 
\[
\mathcal{H}^k (M) \; := \; \big\{ \, \omega\in\Lambda^k TM \;  | \;  d \omega = 0=d^*\omega, \text{  and } \int_{M} \omega \wedge *\omega  <\infty \big\},
\]
where $*$ is the Hodge star operator. Set
\[
b^k_2(M):=\dim_{\IR}\mathcal{H}^k (M).
\]
\end{define}
If $M$ has finite dimensional $L^{2}$-cohomology, then $\mathcal{H}^k (M)$ is isomorphic to the $L^{2}$-cohomology in degree $k$, and then $b^k_2(M)$ is the rank of the $k$-th $L^2$-cohomology group of $M$. Complete finite volume locally symmetric spaces  with $\text{Rank } G = \text{Rank } K$ have finite dimensional $L^{2}$-cohomology, see for example \cite{BorelCass}. Without imposing the equal rank assumption, complete finite volume locally symmetric spaces always have $b_2^k<\infty$, see \cite{BorelGar}.  In this case, $b_2^k(M)$ computes the rank of the {\em reduced} $L^2$-cohomology. Thus, it is interesting to extend Theorem \ref{A} to complete finite volume real- and complex-hyperbolic manifolds. Let $\mu_{\IR}$ and $\mu_{\IC}$ denote respectively the Margulis constants for real- and complex-hyperbolic manifolds with fixed sectional curvature normalization. These constants depend on the dimension, and we refer to Section 8 in the book \cite{BGS} for the definition and basic properties of the Margulis constant of a negatively curved space. We obtain the following results.

\begin{theorem} \label{B}
Let $(M=\Gamma\backslash\textbf{H}^{n}_{\IC}, g_{\IC})$ be a complete finite volume complex-hyperbolic manifold, with $-4\leq sec_{g_{\IC}}\leq -1$. Write $M$ as a disjoint union $M= M_0\cup(\cup_{a}C_a)$, where the $C_a$ are  cusps and where $\inj_{M_{0}}:=\inf_{x\in M_{0}} inj_{g_{i}}(x)\geq \mu_{\IC}$. There exists a constant $c_{n,k}>0$, depending only on dimension and $k$, such that for $k<n$,
$$b_2^k(M) \leq c_{n,k}Vol(M)V_{min}(M_0)^{\frac{k-n}{n}}.$$
\end{theorem}

See Corollary \ref{L2C} for the proof of Theorem \ref{B}.\\

In the real-hyperbolic case we obtain the following bounds. For the analogous bounds in the compact case, we refer to Corollary 116 in \cite{DS17}. 

\begin{theorem} \label{C}
Let $(M^{n}=\Gamma\backslash\textbf{H}^{n}_{\IR}, g_{\IR})$ be a complete finite volume real-hyperbolic manifold, with $sec_{g_{\IR}}=-1$. Write $M$ as a disjoint union $M= M_0\cup(\cup_{a}C_a)$, where the $C_a$ are  cusps and with $\inj_{M_{0}}:=\inf_{x\in M_{0}} inj_{g_{i}}(x)\geq \mu_{\IR}$.  There exists a constant $a_{n,k}>0$ depending only on dimension and $k$ such that for $k<\frac{n-1}{2}$
\[
b_2^k(M) \leq a_{n,k}Vol(M)V_{min}(M_0)^{\frac{2k+1-n}{n-1}}.
\]
Finally if $n=2k+1$, there exists a constant $\alpha(k)>0$ depending only on $k$ such that 
\[
b_2^k(M) \leq \alpha(k)\frac{Vol(M)}{\ln(V_{min}(M_0))}.
\]
\end{theorem}
See Corollary \ref{rhypcusps} and Theorem \ref{criticalreaL} for the proof of Theorem \ref{C}. Interestingly, the proof of Theorem \ref{C} requires estimates beyond the embedded ball barrier, but only on cusps, where the counting of lattice points reduces to the study of lattices in euclidean spaces.  We also apply Theorem \ref{B} and Theorem \ref{C} to towers of coverings associated to a cofinal filtration of the fundamental group of the base hyperbolic manifolds. We refer to Section \ref{towers} for details on the asymptotic behavior of $L^2$-cohomology along such towers.

Let $L(x) := -x_0^2+\sum_{j=1}^nx_j^2$ and $H(z):= -|z_0|^2+\sum_{j=1}^n|z_j|^2.$ Let $G(L)$ and $G(H)$ denote the isometry groups of $L$ and $H$ respectively. Specializing to principal congruence subgroups,  Theorems \ref{B} and \ref{C} reduce to 
\begin{theorem}
For $k<\frac{n-1}{2},$ there exists a constant  $a(n,k) >0$ such that for $\Gamma $ a torsion free principal congruence subgroup of $G(L)(\IZ),$ 
\begin{align}b_2^k(\Gamma \backslash\textbf{H}^{n}_{\IR})\leq a(n,k) Vol(\Gamma \backslash\textbf{H}^{n}_{\IR})^{1-\frac{4(n-1-2k)}{n(n+1)}}.
\end{align}
For $k<n$, there exists a constant  $b(n,k) >0$ such that for $\Gamma  $ a torsion free principal congruence subgroup of $G(H)(\IZ)$,
\begin{align}b_2^k(\Gamma \backslash\textbf{H}^{n}_{\IC})\leq b(n,k) Vol(\Gamma\backslash\textbf{H}^{n}_{\IC})^{1-\frac{4(n-k)}{(n+1)^2-1}}.
\end{align}
\end{theorem}
For principal congruence subgroup quotients of the complex $2$ ball, we thus have $b_2^1$ growing at most like the square root of the volume. 
Surprisingly, this is sharper than both our estimate in Corollary \ref{congruence} and Sarnak and Xue's estimate (but still weaker than Marshall's estimate) for the compact case. The improvement results from our sharper injectivity radius estimates in Subsection \ref{sharp}. In the real hyperbolic case, our estimates extend Yeung's estimates for cocompact lattices \cite[Theorem 2.4.1]{Yeung1} to noncocompact lattices.

Finally, we  study the de Rham cohomology of complete finite volume hyperbolic manifolds with cusps along a cofinal tower. This relies on the topological interpretation of $L^2$-cohomology of locally symmetric varieties (\emph{cf}. \cite{Zuc1}), and on an estimate of the number of cusps along such towers. The problem of estimating the number of cusps in terms of the volume on hyperbolic manifolds with cusps is a well-studied problem in geometric topology; see for example \cite{Kel98}, \cite{Parker}, \cite{DD15}, \cite{BT18} among many other references. Nonetheless, here we need a different point of view on this problem: we  consider the asymptotic behavior of the volume normalized number of cusps along a cofinal tower, and this point of view seems to be new. We refer to Section \ref{tinterpretation} for the background and details.\\
 
\noindent\textbf{Acknowledgments}. We thank Michael Lipnowski for his advice and input on many aspects  of  this work, including expanding the scope of Lemma \ref{thankmike}. We also thank the referees for the many suggestions that helped us improving  the presentation.


\section{The Geometry of Geodesic Balls in Rank One Symmetric Spaces}\label{Carnot}

For for $\IK=\IR, \IC, \IH$, or $\IO$, let  $(\textbf{H}^{n}_{\IK},\textbf{g}_{\IK})$ denote the corresponding $\IK$- hyperbolic spaces.  These spaces have real dimensions respectively $n$, $2n$, $4n$ and $16$. They are symmetric spaces given by the following quotients: 
\begin{align}\label{list1}
&\textbf{H}^{n}_{\IR}=SO(n, 1)/SO(n), \quad \quad \quad \textbf{H}^{n}_{\IC}=SU(n, 1)/U(n), \\ \notag
&\textbf{H}^{n}_{\IH}=Sp(n, 1)/Sp(n)Sp(1), \quad \textbf{H}^{2}_{\IO}=F_{4(-20)}/Spin(9).
\end{align}

Given a point $p\in \textbf{H}^{n}_{\IK}$,  denote by $B_r(p)$ and $S_{r}(p)$ the geodesic ball and sphere of radius $r$ around $p$, respectively. 
In geodesic polar coordinates around the point $p$, 
\begin{align}\label{spherical}
\textbf{g}_{\IK}=dr^{2}+g_{r},
\end{align}
where in the real-hyperbolic case, $g_{\IR}=dr^2+\sinh^2(r)d\sigma^2$, with $d\sigma^2$ the usual round metric on the sphere. In the remaining cases, the metric is best described in terms of a generalized Hopf fibration:
\[
S^{1}\rightarrow S^{2n-1} \rightarrow \textbf{P}^{n-1}_{\IC}, \quad S^{3}\rightarrow S^{4n-1}\rightarrow \textbf{P}^{n-1}_{\IH},
\text{ and }
S^{7}\rightarrow S^{15}\rightarrow\textbf{P}^{1}_{\IO},
\]
where $\textbf{P}^{n-1}_{\IC}$, $ \textbf{P}^{n-1}_{\IH}$ are respectively complex and quaternionic projective spaces; in the octonionic case, we set $\textbf{P}^{1}_{\IO}:= S^{8}$. For convenience, we summarize these fibrations as follows
\[
S^{\dim{\IK}-1}\rightarrow S^{m_{\IK, n}}\rightarrow \textbf{P}^{n}_{\IK},
\]
where $\dim{\IC}=2$, $\dim{\IH}=4$, $\dim{\IO}=8$, and where 
\[
m_{\IC, n}=2n-1, \quad m_{\IH, n}=4n-1, \quad m_{\IO, 1}=15;
\]
in the octonionic case, only $n=1$ occurs. We normalize the max of the sectional curvature of $\textbf{g}_{\IK}$ to be $-1$. The sectional curvature of $\textbf{g}_{\IK}$ is then negative and quarter-pinched:
\[
-4\leq \sec_{\textbf{g}_{\IK}}\leq -1.
\] 
We fix this normalization for the rest of this paper. 

Decompose the metric on the relevant round unit  spheres as $g_{round} = g_v+g_h,$ where 
$g_v$ is the metric restricted to vectors tangent to the generalized Hopf fiber, and $g_h$ is the restriction of the metric to the vectors orthogonal to the fiber. With this notation, we have 
\begin{align}\label{sphericalmetric}
\textbf{g}_{\IK}=dr^{2}+\frac{1}{4}\sinh^2(2r)g_v + \sinh^2(r)g_h.
\end{align}
With this metric, the Hopf fibers are totally geodesic and equipped with the standard round metric up to scale; the metric on the horizontal vectors is the pullback from the  the base of the standard (rescaled) symmetric metric on $\textbf{P}^{n}_{\IK}$. 
Given this explicit description of the  metric, we now compute the second fundamental form of the geodesic spheres.

\begin{proposition}\label{2nd-fundamental}
 The second fundamental form $\textbf{h}_{\IK}(r)$ of a geodesic sphere $S_{r}$ in $(\textbf{H}^{n}_{\IK}, \textbf{g}_{\IK})$ has the  expression:
\begin{equation*} 
\textbf{h}_{\IK}(r)= 2\coth(2r)\frac{\sinh^{2}(2r)}{4}g_v + \coth(r)\sinh^2(r)g_h. 
\end{equation*}
\end{proposition}

\begin{proof}
Recall that the Hessian of the smooth distance function centered at $p$ is proportional to the second fundamental form of a geodesic sphere:
\[
\textbf{h}_{\IK}(r)=\frac{1}{2} \text{Hess}(r)=\frac{1}{2} L_{\partial_{r}}(\textbf{g}_{\IK}),
\]
where $L_{\partial_{r}}$ is the Lie derivative with respect to the unit length radial vector field. We refer to \cite[Proposition 3.2.11]{Petersen} for the derivation of this important identity. Now, we compute
\[
 L_{\partial_{r}}(\textbf{g}_{\IK})=L_{\partial_{r}}(g_{r}), 
\]
so that the proposition follows from \eqref{sphericalmetric}.
\end{proof}

\begin{corollary}\label{mean-curvature}
 For $\IK\not = \IR,$ the second fundamental form $\textbf{h}_{\IK}(r)$ of a geodesic sphere $S_{r}$ in $(\textbf{H}^{n}_{\IK}, \textbf{g}_{\IK})$ has two distinct eigenvalues
\[
\lambda_{1}(r)=2\coth(2r)=\coth(r)+\tanh(r), \quad \lambda_{2}(r)=\coth(r), 
\]
with multiplicities $m(\lambda_{1}, \IK)=\dim(\IK)-1$, $m(\lambda_{2}, \IK)=1-\dim{\IK}+m_{\IK, n}$.  Finally, the mean curvature $\mathcal{H}_{\IK, n}(r)$ of a geodesic sphere $S_{r}$ is 
\begin{align}\label{meanc}
\mathcal{H}_{\IK, n}(r)=(\dim(\IK)-1)(\coth(r)+\tanh(r))+(1-\dim{\IK}+m_{\IK, n})\coth(r).
\end{align}
\end{corollary}

In the  following lemma, we restate Corollary \ref{mean-curvature} in a form convenient for subsequent applications to the Price inequality given in Section \ref{Price}.

\begin{lemma}\label{eigenvalues}
Let 
\begin{align}\label{eigen}
\lambda_{1}(r)=...=\lambda_{\dim(\IK)-1}(r)>\lambda_{\dim(\IK)}(r)=...=\lambda_{m_{\IK, n}}(r)
\end{align}
denote the ordered eigenvalues of $\textbf{h}_{\IK}(r)$. For any integer $1\leq k<\dim_{\IR}\textbf{H}^{n}_{\IK}$, we  have

\begin{equation*} 
\Big(\frac{\mathcal{H}_{\IK, n}(r)}{2}-\sum^{k}_{i=1}\lambda_{i}(r)\Big)= 
\left\{ \begin{array}{rl} &\Big(\frac{m_{\IK, n}}{2}-k\Big)\coth(r)+\Big(\frac{\dim(\IK)-1}{2}-k\Big)\tanh(r), \\
&\text{if } k \leq \dim(\IK)-1;\\
&\\
&\Big(\frac{m_{\IK, n}}{2}-k\Big)\coth(r)-\frac{(\dim(\IK)-1)}{2}\tanh(r), \\
&\text{if } k > \dim(\IK)-1.\\
\end{array} \right. 
\end{equation*}

\end{lemma}

\begin{proof}
It suffices to combine Equations \eqref{meanc} and \eqref{eigen}. 
\end{proof}

\section{On the Betti Numbers of Compact Quotients of $\textbf{H}^{n}_{\IK}$}\label{Price}

In this section, we study Betti numbers of compact locally symmetric rank one spaces via a Price inequality for harmonic forms. We recall the following result for compact real-hyperbolic spaces, which we previously treated in \cite[Corollary 116]{DS17} (see also \cite[Theorem 2.4.1]{Yeung1} for estimates outside the critical degree). Throughout the paper a \emph{cofinal filtration} of a finitely presented group will always be a filtration via finite index normal and nested subgroups.

\begin{theorem}\label{real hyperbolic}  
Let $(M^{n}:=\Gamma\backslash\textbf{H}_{\IR}^{n})$ be a closed real-hyperbolic manifold with $sec_{g_{\IR}}=-1$ and injectivity radius $\inj_M\geq 1$. Given a cofinal filtration $\{\Gamma_{i}\}$ of $\Gamma$,  denote by 
$\pi_{i}: M_{i}\rightarrow M$ the regular Riemannian cover of $M$ associated to $\Gamma_{i}$. 
For any integer $k<\frac{n-1}{2},$  there exists a positive constant $c_{1}(n,k)$ (independent of $\Gamma$ and $\{\Gamma_{i}\}$) such that   
\begin{align}
\frac{b^{n-k}(M_i)}{Vol(M_i)}=\frac{b^{k}(M_{i})}{Vol(M_{i})}\leq c_{1}(n, k)V_{min}(M_i)^{-\frac{n-1-2k}{n-1}}.
\end{align}
In particular, the sub volume growth of the Betti numbers along the tower of coverings is exponential in the injectivity radius. For the critical degree $k=\frac{n-1}{2}$, there exists  a positive constant $c_{2}(n)$ (independent of $\Gamma$ and $\{\Gamma_{i}\}$) such that  
\begin{align}\label{critdeg0}
\frac{b^{k+1}(M_i)}{Vol(M_i)}=\frac{b^{k}(M_{i})}{Vol(M_{i})}\leq \frac{c_{2}(n)}{\inj_{M_{i}}}. 
\end{align} 
\end{theorem}

\begin{remarks}\label{rmk}
Along a cofinal tower one can show that
\[
\inj_{M_{i}}\to\infty, \quad \textrm{as} \quad i\to\infty,
\]
see for example \cite[Theorem 2.1]{DW78}, so that Theorem \ref{real hyperbolic} immediately implies that the normalized Betti numbers
\[
\frac{b^{k}(M_i)}{Vol(M_i)}
\]
go to zero along the tower whenever $k\neq \frac{n}{2}$.
\end{remarks}


For the rest of this section, we focus on  $(\Gamma\backslash\textbf{H}^{n}_{\IK}, g_{\IK})$ with $\IK=\IC, \IH$, or $\IO$. For the remainder of the paper $\Gamma$  will always denote a torsion free  discrete subgroup of the isometries of $\textbf{H}^{n}_{\IK}$.  Let  $\inj_{\Gamma}$ denote the injectivity radius of $(\Gamma\backslash\textbf{H}^{n}_{\IK}, g_{\IK})$.

Given a $1$-form $\phi$, let $e(\phi)$ denote exterior multiplication on the left by $\phi$, and let $e^*(\phi)$ denote its adjoint operator. Let $\rho$ be a smooth  function with $|d\rho|=1$ and let $\frac{\p}{\p \rho}:= \nabla \rho.$ Given a local orthonormal frame $\{\frac{\p}{\p \rho}\}\cup \{V_j\}_j$ and coframe $\{d\rho\}\cup \{\omega^j\}_j$, recall that the exterior derivative can be written $d= e(d\rho)\nabla_{\frac{\p}{\p \rho}}+ e(\omega^j)\nabla_{V_j}$, and hence, acting on forms,  the Lie derivative in the $\frac{\p}{\p \rho}$ direction can be written as 
\begin{align}\label{lie}L_{\p_\rho}=\{d,e^*(d\rho)\} = \nabla_{\frac{\p}{\p \rho}}+e(\omega^j)e^*(\nabla_{V_j}d\rho).
\end{align} 
Fix a point $p\in M$, and consider a geodesic ball $B_{R}(p)$, with $0<R\leq\inj_{\Gamma}$. 
The Hopf fibers are framed; so, we will henceforth choose our orthonormal frame so that for $1\leq j\leq \dim(\IK)-1,$ $V_j$ is tangent to the fiber and globally defined in $B_R(p)$. We call such a frame and coframe {\em adapted}. 
Next, for any harmonic $k$-form $h$ on $\Gamma\backslash\textbf{H}_{\IK}^{n}$, we  define $\dim(\IK)$ auxiliary nonnegative  functions that naturally arise in our Price equality.   Set
\begin{align}\label{mu0}
\mu_{h}(r):=\frac{\int_{S_{r}}|e^{*}(dr)h|^{2}d\sigma}{\int_{S_{r}}|h|^{2}d\sigma},
\end{align}
and
\begin{align}\label{zetas}
\zeta^{j}_{h}(r):=\frac{\int_{S_{r}}|e^{*}(\omega_{j})h|^{2}d\sigma}{\int_{S_{r}}|h|^{2}d\sigma},
\end{align}
for $j\in 1, ..., \dim(\IK)-1$, where $d\sigma$ is the Riemannian measure induced on the geodesic sphere $S_{r}$ by $g_{\IK}$. These functions are by definition non negative, bounded from above by one, and  well defined for any $0<r\leq\inj_{\Gamma}$. 
(In fact, we only need $\sum_j\zeta^j_h$ in applications; so, it is not really necessary that the tangent space to the fiber be framed.) \\

With this notation, we can now state and prove our first lemma.

\begin{lemma}\label{K-first}
 Given  $h\in\mathcal{H}^{k}(\Gamma\backslash\textbf{H}_{\IK}^{n})$, $p\in \Gamma\backslash\textbf{H}_{\IK}^{n}$, and  $R\in(0, \inj_{\Gamma})$, we have
\begin{align}\notag
\int_{S_{R}(p)}\Big(\frac{1}{2}-\mu_{h}(R)\Big)|h|^{2}d\sigma&=\int_{B_{R}(p)}q_{h}(r)|h|^{2}dv
\end{align}
with
\[
q_{h}(r)=\frac{\mathcal{H}_{\IK, n}(r)}{2}-k\coth(r)-\tanh(r)\sum^{\dim(\IK)-1}_{j=1}\zeta^{j}_{h}(r)+\mu_{h}(r)\coth(r),
\]
with $\mathcal{H}_{\IK, n}(r)$ as in Equation \eqref{meanc}.
\end{lemma} 

\begin{proof}

In an orthonormal adapted coframe $dr, \omega^{1}, ..., \omega^{m_{\IK, n}}$, we have by Proposition \ref{2nd-fundamental}  
\begin{align}\label{refereeannoince}\notag
L_{\partial_{r}}&=\nabla_{\partial_r}+e(\omega^{j})e^{*}(\nabla_{V_{j}}dr) \\ \notag
&=\nabla_{\partial_r}+(\coth(r)+\tanh(r))\sum^{\dim(\IK)-1}_{j=1}e(\omega^{j})e^{*}(\omega^{j})+\coth(r)\sum^{m_{\IK, n}}_{j=\dim(\IK)}e(\omega^{j})e^{*}(\omega^{j})\\ \notag
&=\nabla_{\partial_r}+\coth(r)\big\{e(dr)e^{*}(dr)+\sum^{m_{\IK, n}}_{j=1}e(\omega^{j})e^{*}(\omega^{j})\big\}\\ 
&-\coth(r)e(dr)e^{*}(dr)+\tanh(r)\sum^{\dim(\IK)-1}_{j=1}e(\omega^{j})e^{*}(\omega^{j}).
\end{align}
See page 6 in \cite{DS17} for more details of the computation. Using now the characterization $L_{\p_\rho}=\{d,e^*(d\rho)\}$ and the harmonicity of $h$,  Stokes' Theorem yields
\begin{align}\notag
\int_{B_{R}(p)}\langle L_{\partial_{r}}h, h\rangle dv&=\int_{S_{R}(p)}|e^*(dr)h|^{2}d\sigma. \\ \notag
\end{align}
On the other hand, by Equation \ref{refereeannoince} 
\begin{align}\notag
\int_{B_{R}(p)}\langle L_{\partial_{r}}h, h\rangle dv&=\int_{B_{R}(p)}\langle\nabla_{\partial_{r}}h, h\rangle dv+k\int_{B_{R}(p)}\coth(r)|h|^{2}dv \\ \notag
&-\int_{B_{R}(p)}\mu_{h}(r)\coth(r)|h|^{2}dv+\sum^{\dim(\IK)-1}_{j=1}\int_{B_{R}(p)}\zeta^{j}_{h}(r)\tanh(r)|h|^{2}dv,\\ \notag
&=-\int_{B_{R}(p)}\frac{\mathcal{H}_{\IK, n}(r)}{2}|h|^{2} dv+\frac{1}{2}\int_{S_{R}(p)}|h|^{2}d\sigma+k\int_{B_{R}(p)}\coth(r)|h|^{2}dv \\ \notag
&-\int_{B_{R}(p)}\mu_{h}(r)\coth(r)|h|^{2}dv+\sum^{\dim(\IK)-1}_{j=1}\int_{B_{R}(p)}\zeta^{j}_{h}(r)\tanh(r)|h|^{2}dv.\\ \notag
\end{align}
We  rearrange these terms to obtain
\begin{align}\notag
&\int_{S_{R}(p)}\Big(\frac{1}{2}|h|^{2}-|e^*(dr)h|^{2}\Big)d\sigma \\ \notag
&=\int_{B_{R}(p)}\frac{\mathcal{H}_{\IK, n}(r)}{2}|h|^{2} dv-k\int_{B_{R}(p)}\coth(r)|h|^{2}dv \\ \notag
&+\int_{B_{R}(p)}\mu_{h}(r)\coth(r)|h|^{2}dv-\sum^{\dim(\IK)-1}_{j=1}\int_{B_{R}(p)}\zeta^{j}_{h}(r)\tanh(r)|h|^{2}dv,\\ \notag
\end{align}
and the lemma  follows from Equations \eqref{mu0} and \eqref{zetas}.
\end{proof}

We now  control the positivity of the geometric term $q_{h}(r)$ appearing in Lemma \ref{K-first}.

\begin{lemma}\label{K-second}
Given $h\in\mathcal{H}^{k}(\Gamma\backslash\textbf{H}_{\IK}^{n})$, $p\in \Gamma\backslash\textbf{H}_{\IK}^{n}$, and $r\in(0,\inj_{\Gamma})$, we have
\begin{equation*} 
\left\{ \begin{array}{rl} & \boxed{q_{h}(r)>\frac{\dim(\textbf{H}^{n}_{\IK})-\dim(\IK)}{2}-k\geq 1,} \\
& \\
&\text{if } \dim(\IK)-1<k<\frac{\dim(\textbf{H}^{n}_{\IK})-\dim(\IK)}{2};\\
&\\
& \boxed{q_{h}(r)>\frac{\dim(\textbf{H}^{n}_{\IK})+\dim(\IK)-2}{2}-2k\geq 1,} \\
& \\
&\text{if } k\leq\dim(\IK)-1,\\
& \\
&\text{unless } \\
&\IK=\IC, \quad n=2, \quad k=1, \\
&\IK=\IH, \quad n=2, \quad k=3, \\
&\IK=\IO, \quad n=1,\quad k=6, 7.
\end{array} \right. 
\end{equation*}
In particular, in all of these cases we have $\mu_{h}(r)<1/2$ for any $0<r< \inj_{\Gamma}$.
\end{lemma} 

\begin{proof}
First observe that for any $0<r< \inj_{\Gamma}$, $q_{h}(r)$  satisfies
\[
q_{h}(r)\geq\Big(\frac{\mathcal{H}_{\IK, n}(r)}{2}-\sum^{k}_{i=1}\lambda_{i}(r)\Big),
\]
 where the $\{\lambda_{i}\}_{i}$ are the ordered eigenvalues of the second fundamental form of the geodesic sphere $S_{r}$ as in \eqref{eigen}. We then conclude using Lemma \ref{eigenvalues}.
\end{proof}

We can now state a general Price inequality for rank one locally symmetric spaces.

\begin{lemma}\label{K-Price}
Let $\inj_{\Gamma}\geq 1 $. There exists a constant $d_{(\IK,n, k)}>0$ such that for $h\in\mathcal{H}^{k}(\Gamma\backslash\textbf{H}_{\IK}^{n})$ and $p\in \Gamma\backslash\textbf{H}_{\IK}^{n}$, 
\begin{equation*} 
\left\{ \begin{array}{rl} & \boxed{\int_{B_{1}(p)}|h|^{2}dv\leq d_{(\IK, n, k)}V_{min}(\Gamma\backslash\textbf{H}_{\IK}^{n})^{-\frac{(\dim(\textbf{H}^{n}_{\IK})-\dim(\IK)-2k)\inj_{\Gamma}}{\dim(\textbf{H}^{n}_{\IK})+\dim(\IK)-2}}\int_{B_{\inj_{\Gamma}}(p)}|h|^{2}dv,} \\
& \\
&\text{if } \dim(\IK)-1<k<\frac{\dim(\textbf{H}^{n}_{\IK})-\dim(\IK)}{2};\\
&\\
& \boxed{\int_{B_{1}(p)}|h|^{2}dv\leq d_{(\IK, n, k)}V_{min}(\Gamma\backslash\textbf{H}_{\IK}^{n})^{-\frac{(\dim(\textbf{H}^{n}_{\IK})+\dim(\IK)-4k)\inj_{\Gamma}}{\dim(\textbf{H}^{n}_{\IK})+\dim(\IK)-2}}\int_{B_{\inj_{\Gamma}}(p)}|h|^{2}dv,} \\
& \\
&\text{if } k\leq\dim(\IK)-1,\\
& \\
&\text{unless } \\
&\IK=\IC, \quad n=2, \quad k=1, \\
&\IK=\IH, \quad n=2, \quad k=3, \\
&\IK=\IO, \quad n=1,\quad k=6, 7.
\end{array} \right. 
\end{equation*}
\end{lemma}

\begin{proof}

The integral equality for harmonic forms of Proposition 16 in \cite{DS17} with $\sigma =1$ implies
\[
\phi_h(\tau)\int_{B_{\tau}(p)}q_{h}(s)|h|^{2}dv=\int_{B_{1}(p)}q_{h}(s)|h|^{2}dv\geq \int_{B_{1}(p)}|h|^{2}dv
\]
with 
\[
\phi_h(\tau)=e^{-\int^{\tau}_{1}\frac{q_h(s)ds}{\frac{1}{2}-\mu_{h}(s)}}.
\]
By Lemma \ref{K-second}, we know that $\mu_{h}(s)<1/2$ for any $s\leq\inj_{\Gamma}$. Integrate the lower bounds on $q_{h}(s)$ given in Lemma \ref{K-second} to obtain the desired inequality. 
\end{proof}

We now use Lemma \ref{K-Price} to prove an effective bound in most degrees for normalized Betti numbers of rank one compact locally symmetric spaces. 

\begin{theorem}\label{K-theorem}
There exists a positive constant $c_{(\IK, n, k)}$ such that if $\Gamma$ is cocompact and $\inj_{\Gamma}\geq 1$, 
\begin{equation*} 
\left\{ \begin{array}{rl} & \boxed{\frac{b^{k}(\Gamma\backslash	\textbf{H}_{\IK}^{n})}{Vol(\Gamma\backslash			\textbf{H}_{\IK}^{n})}\leq c_{(\IK, n, k)}V_{min}(\Gamma\backslash\textbf{H}_{\IK}^{n})^{-\frac{(\dim(\textbf{H}^{n}_{\IK})-\dim(\IK)-2k)\inj_{\Gamma}}{\dim(\textbf{H}^{n}_{\IK})+\dim(\IK)-2}},} \\
& \\
&\text{if } \dim(\IK)-1<k<\frac{\dim(\textbf{H}^{n}_{\IK})-\dim(\IK)}{2};\\
&\\
& \boxed{\frac{b^{k}(\Gamma\backslash	\textbf{H}_{\IK}^{n})}{Vol(\Gamma\backslash\textbf{H}_{\IK}^{n})}\leq c_{(\IK, n, k)}V_{min}(\Gamma\backslash\textbf{H}_{\IK}^{n})^{-\frac{(\dim(\textbf{H}^{n}_{\IK})+\dim(\IK)-4k)\inj_{\Gamma}}{\dim(\textbf{H}^{n}_{\IK})+\dim(\IK)-2}},} \\
& \\
&\text{if } k\leq\dim(\IK)-1,\\
& \\
&\text{unless } \\
&\IK=\IC, \quad n=2, \quad k=1, \\
&\IK=\IH, \quad n=2, \quad k=3, \\
&\IK=\IO, \quad n=1,\quad k=6, 7.
\end{array} \right. 
\end{equation*}

\end{theorem}
\begin{proof}
Given the inequalities in Lemma \ref{K-Price}, the estimates on the Betti numbers follow by combining Lemma 45 and Lemma 51 in \cite{DS17}. 
\end{proof}

Theorem \ref{K-theorem} sharpens the bounds for negatively curved and quarter-pinched manifolds given by the general theory developed in Section 7 of \cite{DS17}. Indeed, Theorem \ref{K-theorem} relies heavily upon the special properties of the locally symmetric metrics on rank one locally symmetric spaces.  In the next section, we push this analysis further by exploiting the K\"ahler properties of the complex-hyperbolic metric. 

\section{Refined Estimates for Compact Quotients of $\textbf{H}^{n}_{\IC}$ }\label{Complex Hyperbolic}

In this section, we improve the results of Section \ref{Price} in the case of complex-hyperbolic manifolds. These further results rely on the fact that, in this case, the  metric is not only symmetric but also K\"ahler. 
If we regard $\textbf{H}^{n}_{\IC}$ as the unit ball $\IB^{n}$ in $\IC^{n}$, the K\"ahler form is 
\begin{align}\label{Bergman}
\omega_{\IC}=\frac{i}{2}\overline{\partial}\partial\log(1- |z|^{2}).
\end{align}
The associated Riemannian metric is exactly the metric $\textbf{g}_{\IC}$   discussed in Section \ref{Carnot}. 
Given a compact complex-hyperbolic manifold $\Gamma\backslash\textbf{H}^{n}_{\IC}$, we denote by $J$ its complex structure operator. Let $h$ be a harmonic $k-$form on $\Gamma\backslash\textbf{H}^{n}_{\IC}$. 
 Since $\Gamma\backslash\textbf{H}^{n}_{\IC}$ is K\"ahler, $Jh$ is also harmonic with 
\[
|h|(p)=|Jh|(p),
\]
for any $p$. In a geodesic ball, $J\frac{\p}{\p r}$ is tangent to the Hopf fiber. Hence we can choose the adapted coframe used in the definition \eqref{zetas} with $\omega^1 = Jdr.$  
Thus, we find convenient to set

\begin{align}\label{defzeta}
\zeta_{h}(r):=\frac{\int_{S_{r}(p)}|e^*(Jdr)h|^{2}d\sigma}{\int_{S_{r}(p)}|h|^{2}d\sigma},
\end{align}
where $\zeta_{h}(r)=\zeta^{1}_{h}(r)$ in our prior notation.
Finally, since 
\[
|e^{*}(dr)(Jh)|=|e^{*}(Jdr)(h)|, \text{  and}\quad |e^{*}(Jdr)(Jh)|=|e^{*}(dr)h|,
\]
we conclude that 
\begin{align}\label{mu-zeta}
\mu_{Jh}=\zeta_{h}, \quad \zeta_{Jh}=\mu_{h}.
\end{align}

We  now derive a strengthened Price equality for complex-hyperbolic manifolds. We start with the following integration by parts formula.\\ 

\begin{lemma}\label{complex first}Let $n\geq 2$. 
  Given  $h\in\mathcal{H}^{k}(\Gamma\backslash\textbf{H}^{n}_{\IC})$ with $k<n$, $p\in \Gamma\backslash\textbf{H}^{n}_{\IC}$, and  $R<\inj_{\Gamma}$, we have
\begin{align}\notag
\int_{S_{R}(p)}\big(1-\mu_{h}(R)-\mu_{Jh}(R)\big)|h|^{2}d\sigma&=\int_{B_{R}(p)}q(r)|h|^{2}dv
\end{align}
with $\mu_{h}(R)=\mu_{Jh}(R)<1/2$, and
\[
q(r)> 2(n-k-1\big)\coth(r)+2\tanh(2r).
\]
\end{lemma} 

\begin{proof}

For $R\leq\inj_{\Gamma}$, apply Lemma \ref{K-first} to both $h$ and $Jh$ in $\mathcal{H}^{k}(\Gamma\backslash\textbf{H}^{n}_{\IC})$ to obtain
\begin{align}\label{1st}
\int_{S_{R}(p)}\Big(\frac{1}{2}-\mu_{h}(R)\Big)|h|^{2}d\sigma&=\int_{B_{R}(p)}\Big(\frac{\mathcal{H}_{\IC, n}(r)}{2}-k\coth(r)\Big)|h|^{2}dv\\ \notag
&+\int_{B_{R}(p)}\mu_{h}(r)\coth(r)|h|^{2}dv\\ \notag
&-\int_{B_{R}(p)}\mu_{Jh}(r)\tanh(r)|h|^{2}dv,
\end{align}
and
\begin{align}\label{2nd}
\int_{S_{R}(p)}\Big(\frac{1}{2}-\mu_{Jh}(R)\Big)|Jh|^{2}d\sigma&=\int_{B_{R}(p)}\Big(\frac{\mathcal{H}_{\IC, n}(r)}{2}-k\coth(r)\Big)|Jh|^{2}dv\\ \notag
&+\int_{B_{R}(p)}\mu_{Jh}(r)\coth(r)|Jh|^{2}dv \\ \notag
&-\int_{B_{R}(p)}\mu_{h}(r)\tanh(r)|Jh|^{2}dv.
\end{align}
Since $|h|^{2}=|Jh|^{2}$, summing Equations \eqref{1st} and \eqref{2nd} yields
\begin{align}\label{tw9}
\int_{S_{R}(p)}&\big(1-\mu_{h}(R)-\mu_{Jh}(R)\big)|h|^{2}d\sigma\\ \notag
&=\int_{B_{R}(p)}\big((2n-2k-1)\coth(r)+\tanh(r)\big)|h|^{2}dv \\ \notag
&+\int_{B_{R}(p)}\big((\mu_{h}(r)+\mu_{Jh}(r))\coth(r)-(\zeta_{h}(r)+\zeta_{Jh}(r))\tanh(r)\big)|h|^{2}\\ \notag
&=\int_{B_{R}(p)}\big((2n-2k-1)\coth(r)+\tanh(r)\big)|h|^{2}dv \\ \notag
&+\int_{B_{R}(p)}(\mu_{h}(r)+\mu_{Jh}(r))(\coth(r)-\tanh(r))|h|^{2}dv\\ \notag
&>\int_{B_{R}(p)}\big((2(n-k-1))\coth(r)+2\coth(2r)\big)|h|^{2}dv. \\ \notag
\end{align}
The positivity of the right hand side of \eqref{tw9} for nonzero $h$ implies $\mu_h+\mu_{Jh}< 1$. 
To see that $\mu_h = \mu_{Jh}$, we now take the difference of equations \eqref{1st} and \eqref{2nd} to obtain 
\begin{align}\label{diff}
\int_{S_{R}(p)}\Big(\mu_{Jh}-\mu_{h} \Big)|h|^{2}d\sigma&= \int_{B_{R}(p)}(\mu_{h}-\mu_{Jh})2\coth(2r)|h|^{2}dv.
\end{align}
Set $F:= \int_{B_{R}(p)}(\mu_{h}-\mu_{Jh})2\coth(2r)|h|^{2}dv.$ Then we can rewrite \eqref{diff} as 
\begin{align}\label{diffeq}
F'&= -2\coth(2r)F.
\end{align}
Hence 
\begin{align}\label{integrate}
  F(r)    =    \frac{\sinh(2\sigma)}{\sinh(2r)}F(\sigma). 
\end{align}
Since $F(0) = 0$, \eqref{integrate} requires $F(\sigma) = 0$ for all $\sigma$, and $\mu_h = \mu_{Jh}<\frac{1}{2}$ follows.  
\end{proof}

Following Proposition 16 in \cite{DS17}, we have the following Price inequality for complex-hyperbolic manifolds.

\begin{lemma}\label{PriceC}
Let  $inj_\Gamma\geq 1$. There exists a constant $d(n,k)>0$ such that if $h\in\mathcal{H}^{k}(\Gamma\backslash\textbf{H}^{n}_{\IC})$ with $k<n$,  $p\in \Gamma\backslash\textbf{H}^{n}_{\IC}$, and $\tau<\inj_\Gamma$, 
then
\begin{align}\label{ninja}
\int_{B_{1}(p)}|h|^{2}dv\leq d(n, k)e^{-2(n-k)\tau}\int_{B_{\tau}(p)}|h|^{2}dv.
\end{align}
\end{lemma}
\begin{proof}
We have that
\[
\phi(\tau)\int_{B_{\tau}(p)}q(s)|h|^{2}dv=\int_{B_{1}(p)}q(s)|h|^{2}dv\geq \int_{B_{1}}|h|^{2}dv,
\]
with 
\[
\phi(\tau)\leq e^{-\int^{\tau}_{1}q(s)ds}.
\]
Integrating the lower bound on $q(s)$ given in Lemma \ref{complex first}, we  obtain the desired bound.  
\end{proof}

We can now give an effective bound for the Betti numbers of compact complex-hyperbolic manifolds.

\begin{corollary}\label{cor1}
Let $ \Gamma$ be co-compact, with  $\inj_{\Gamma} \geq 1.$ For any positive integer $k<n$,   
\[
\frac{b^{k}(\Gamma\backslash\textbf{H}^{n}_{\IC})}{Vol(\Gamma\backslash\textbf{H}^{n}_{\IC})}\leq d(n, k)V_{min}(\Gamma\backslash\textbf{H}^{n}_{\IC})^{\frac{k-n}{n}},
\]  
where $d(n, k)$ is a positive constant depending only on the dimension and the degree $k$. 
\end{corollary}

\begin{proof}
Given the Price inequality in Lemma \ref{PriceC}, the Betti number estimate follows from the peaking and Moser iteration arguments given in Section 5 of \cite{DS17} (\emph{cf}. Lemma 45 and Lemma 51 in \cite{DS17}). 
\end{proof}



Fix now a cocompact arithmetic group $\Gamma_0$ of a semisimple Lie group $G\subset GL(N,\IR)$, some $N$. Let $\Gamma_q$ be a level $q$ congruence subgroup of $\Gamma_0$ and $K$ a maximal compact subgroup of $G$. Sarnak and Xue \cite[Equation 17]{Sarnak} show that there exists $\delta_{\Gamma_0}$ such that the injectivity radius of $\Gamma_q\backslash G/K$   satisfies 
\begin{align}
\inj_{\Gamma_q} \geq 2\ln(q)-\delta_{\Gamma_{0}}.
\end{align} 
Returning now to $G= SU(n,1)$, we rewrite this in a less scale dependent  manner  as 
\begin{align}\notag
Vol(B_{\inj_{\Gamma_q}})\geq \omega_nq^{2n} + O(q^{(2n-1)}),
\end{align}
for some constant $\omega_n$ depending on the dimension only. By \cite[Equation 22]{Sarnak}, for $\epsilon>0$,  there exists a constant $C_\epsilon >0$, independent of $q$ such that 
\begin{align}\notag
Vol(\Gamma_q\backslash  \textbf{H}^{n}_{\IC})\leq C(\epsilon,\Gamma_0) q^{n^2+2n+\epsilon}. 
\end{align}
Hence for $q$ large, we have for some $c(\epsilon,\Gamma_0) >0$, independent of $q$: 
\begin{align}\notag
V_{min}(\Gamma_q\backslash\textbf{H}^{n}_{\IC}))\geq c(\epsilon,\Gamma_0) Vol(\Gamma_q\backslash  \textbf{H}^{n}_{\IC})^{\frac{2}{n +2 }-\epsilon}. 
\end{align}
Consequently, 
\begin{align}\label{coro3}
b^k(\Gamma_q\backslash  \textbf{H}^{n}_{\IC}) \leq d(k,n)c(\epsilon,\Gamma_0)  Vol(\Gamma_q\backslash  \textbf{H}^{n}_{\IC})^{\frac{n^2+2k+\epsilon}{n^2+2n}}  .
\end{align}
Dividing each side of \eqref{coro3} by $vol(M)$ yields Corollary \ref{congruence}.
When $n=2$, we have 
\begin{align}\notag 
b^1(\Gamma_q\backslash  \textbf{H}^{2}_{\IC}) \leq d(1,2)c(\epsilon,\Gamma_0)    Vol(\Gamma_q\backslash  \textbf{H}^{2}_{\IC})^{\frac{3}{4 } +\epsilon}  .
\end{align}
 
\section{Peaking Revisited}\label{previsited}

We now shift our attention to complete finite volume hyperbolic manifolds. As a first step towards treating cusps, we reformulate the peaking equality of \cite{DS17}, embedding it in an infinite family of equalities that can be used to estimate Betti numbers. We first record an elementary calculus lemma. 

\begin{lemma}\label{linear}
Let $\phi:\IR^b\to\IR$ be a nonzero linear function. Then 
\begin{align}
b = \frac{\max\{|\phi(h)|^2: h\in S^{b-1}\}}{\int_{S^{b-1}}|\phi(h)|^2d\sigma_h},
\end{align}
where $d\sigma_h$ is the usual Riemannian measure on the unit sphere, renormalized to give the sphere unit volume.
\end{lemma}

\begin{proof}
Rotate coordinates so that $\phi(h) = ah_1$. We then have 
\begin{align*}
\frac{\max\{|\phi(h)|^2: h\in S^{b-1}\}}{\int_{S^{b-1}}|\phi(h)|^2d\sigma_h} = \frac{a^2}{\int_{S^{b-1}}a^2h_1^2d\sigma_h} = b,
\end{align*}
and the proof is complete.
\end{proof}

Let $(M, g)$ be a Riemannian manifold. The unit sphere bundle
\[
\pi:S(\Lambda^kTM)\to M
\]  
of $\Lambda^kTM$ inherits a canonical metric from the metric on $M$, and a choice of a metric on the spherical fiber. We pick the standard symmetric metric on the spherical fiber and scale it to have volume $1$. Let $d\bar{v}$ denote the associated Riemannian measure on $S(\Lambda^kTM)$. We then have:
\[
Vol(S(\Lambda^kTM))=\int_{S(\Lambda^kTM)}1 d\bar{v}=\int_{M}1 dv=Vol(M).
\]
Let  $S(\mathcal{H}^k(M))$ denote the unit sphere in $\mathcal{H}^k(M)$, where $\mathcal{H}^k(M)$ is endowed with the Hilbert space structure given by the $L^2$-inner product on harmonic $k-$ forms. 

\begin{corollary}\label{linearh} 
For each $\bar{p}\in S(\Lambda^kTM)$ such that $h(\bar{p})\not = 0$ for some $h\in \mathcal{H}^k(M),$ we have 
\begin{align}\label{rhs}
b^k_2(M):= dim_{\IR}(\mathcal{H}^k(M)) = \frac{\max\{|h(\bar{p})|^2: h\in S(\mathcal{H}^k(M))\}}{\int_{S(\mathcal{H}^k(M))}|h(\bar{p})|^2d\sigma_h},
\end{align}
where $d\sigma_h$ is the usual Riemannian measure on the unit sphere, renormalized to give the sphere unit volume.
\end{corollary}
\begin{proof}
Given such $\bar{p}\in S(\Lambda^kTM)$, consider the associated non-zero linear function obtained by evaluating $h\in\mathcal{H}^k(M)$ at $\bar{p}$. Apply Lemma \ref{linear} to this linear function, and the proof is complete. 
\end{proof}



Next, observe that the $L^2$-norm on $k$-forms satisfies for all measurable $\Omega\subset M$, 
\begin{align}\label{zero}
\int_{\Omega}|h|^2(p)dv_p=\binom{n}{k} \int_{\pi^{-1}(\Omega)}|h(\bar{p})|^2d\bar{v}_{\bar{p}},
\end{align}
as the volume of the fibers is one. Using \eqref{zero}, we write
\begin{align}\label{one}
1& =\int_{M}\int_{S(\mathcal{H}^k(M))}|h|^2(p)dv_{p} d\sigma_h=\binom{n}{k}\int_{S(\Lambda^k TM)}\int_{S(\mathcal{H}^k(M))}|h(\bar{p})|^2 d\sigma_h d\bar{v}_{\bar{p}};
\end{align}
hence the average value of $\int_{S(\mathcal{H}^k(M))}|h|^2(p)d\sigma_h$ and $\binom{n}{k}\int_{S(\mathcal{H}^k(M))}|h(\bar{p})|^2d\sigma_h$  is $\frac{1}{Vol(M)}$.\\  

When working on manifolds with cusps, we would like for some $\bar{p}\in S(\Lambda^kTM)$ for which 
\[
\binom{n}{k}\int_{S(\mathcal{H}^k(M))}|h(\bar{p})|^2d\sigma_h\geq \frac{1}{Vol(M)}
\] 
  to live above the thick part of the manifold. In fact, when this is the case, Corollary \ref{linearh} implies
\[
b^k_2(M)\leq\max\{| h(\bar{p})|^2: h\in S(\mathcal{H}^k(M))\}\binom{n}{k}Vol(M),
\] 
and as $|h(\bar{p})|^2\leq |h|^2(\pi(\bar{p}))$, we obtain an upper bound on $b^k_2(M)$ by using a Price inequality on geodesic balls and a Moser iteration argument. We refer to Sections \ref{ccusps} and \ref{rlattices} for the implementation of this strategy on finite volume hyperbolic manifolds with cusps. 

On the other hand, if there is no such $p$, we obtain additional information, which we record with a lemma. 

\begin{lemma}\label{split}
Let $S(\Lambda^kTM)= \pi^{-1}(\Omega_1\cup \Omega_2)$, where $\Omega_1$ and $\Omega_2$ are measurable subsets of $M$. Given $\alpha\in (0,1)$, suppose that 
\[
\binom{n}{k}\int_{\pi^{-1}(\Omega_{1})}\int_{S(\mathcal{H}^k(M))}|h(\bar{p})|^2d\sigma_hd\bar{v}_{\bar{p}} <\frac{\alpha Vol(\Omega_1)}{Vol(M)}.
\]
Then, the average value of $\binom{n}{k}\int_{S(\mathcal{H}^k(M))}|h(\bar{p})|^2d\sigma_h$ for $\bar{p}\in \pi^{-1}(\Omega_2)$ is greater than or equal to 
$\frac{1-\alpha}{Vol(\Omega_2)}.$ 
\end{lemma}

\begin{proof}
From Equation \eqref{one}, we have
\begin{align}
1&= \int_{S(\Lambda^kTM)}\binom{n}{k}\int_{S(\mathcal{H}^k(M))}|h(\bar{p})|^2d\sigma_hd\bar{v}_{\bar{p}}\nonumber\\
&\leq 
\int_{\pi^{-1}(\Omega_1)}\binom{n}{k}\int_{S(\mathcal{H}^k(M))}|h(\bar{p})|^2d\sigma_hd\bar{v}_{\bar{p}}+\int_{ \pi^{-1}(\Omega_2)}\binom{n}{k}\int_{S(\mathcal{H}^k(M))}|h(\bar{p})|^2d\sigma_hd\bar{v}_{\bar{p}}\nonumber\\
&\leq\alpha \frac{Vol(\Omega_1)}{Vol(M)}+\int_{ \pi^{-1}(\Omega_2)}\binom{n}{k}\int_{S(\mathcal{H}^k(M))}|h(\bar{p})|^2d\sigma_hd\bar{v}_{\bar{p}}.
\end{align}
Hence 
\[
\int_{\pi^{-1}(\Omega_2)}\binom{n}{k}\int_{S(\mathcal{H}^k(M))}|h(\bar{p})|^2d\sigma_hd\bar{v}_{\bar{p}}\geq (1-\alpha),
\]
and the result follows. 
\end{proof}

We will use this lemma to balance the relatively small volume of the cusp regions against the smaller injectivity radius at a point in a cusp. 

\section{$L^{2}$-Harmonic Forms on Real-Hyperbolic Cusps}\label{standard}

Cusps of complete finite volume real-hyperbolic manifolds are isometric to $[0,\infty)\times F$, equipped with the metric $ds^2+e^{-2s}g_{F}$, where $F$ is a compact quotient of the $(n-1)$-dimensional Euclidean space, and $g_{F}$ is a flat metric on $F$. For simplicity, for most of this section, we  consider cusps of complete finite volume real-hyperbolic manifolds that are isometric to $[0,\infty)\times T^{n-1}$, equipped with the metric $ds^2+e^{-2s}g_{T^{n-1}}$, where $T^{n-1}$ is an $(n-1)$-dimensional real torus, and $g_{T^{n-1}}$ is a flat metric on $T^{n-1}$. This is always the case, if for example, the manifold is the quotient of hyperbolic space by a {\em neat } arithmetic subgroup. The Price inequality for general cusps follows immediately from this special case, as we show at the end of this section. (Moreover, given a  complete finite volume real-hyperbolic manifold, one can always pass to a finite cover where all of the cusps have tori cross sections, see for example \cite{Hummel}. If a base of a tower of manifolds has this property, then so does every manifold in the tower.)\\ 

Given a covariant constant orthonormal frame  and coframe on $T_0$, extend it to a $\nabla_{\frac{\p}{\p s}}$ covariant constant frame $\{\frac{\p}{\p s}\}\cup \{X_{i} \}^{n-1}_{i=1}$ and coframe $\{ds\}\cup \{\omega^i \}^{n-1}_{i=1}$ on the cusp. Using the explicit form of the hyperbolic metric on the cusp and the flatness of $g_{T^{n-1}}$, it is straightforward to verify that:
\[
\nabla_{\partial_s}\partial_s=0, \quad \nabla_{X_i}X_j= \delta_{ij}\frac{\p}{\p s}, \quad \nabla_{X_i} \frac{\p}{\p s}=-X_i.
\]
On differential forms, we have therefore  the following identity
\begin{align}\notag
L_{\partial_{s}}&=\nabla_{\partial_s}+\sum^{n-1}_{i=1}e(\omega_{j})e^{*}(\nabla_{X_{j}}ds)\\ \notag
&=\nabla_{\partial_s}-\sum^{n-1}_{i=1}e(\omega_{j})e^{*}(\omega_{j})-e(ds)e^*(ds)+e(ds)e^{*}(ds), \\ \notag
\end{align}
so that the pointwise inner product $(L_{\partial_s}\alpha, \alpha)$ satisfies
\begin{align}\label{Lie}
(L_{\partial_s}\alpha, \alpha)=(\nabla_{\partial_{s}}\alpha, \alpha)-k|\alpha|^{2}+|i_{\partial_s}\alpha|^{2}.
\end{align}
Given a strongly harmonic $L^{2}$ $k$-form $\alpha$ on a  cusp, define for $s\in[0, \infty)$
\begin{align}\label{mu}
\mu_\alpha(s):=\frac{\int_{T_s}|i_{\partial_s}\alpha|^{2}d\sigma_s}{\int_{T_s}|\alpha|^{2}d\sigma_{s}}.
\end{align}
For  $0\leq s<K$,  define $\Omega_{uK}$ to be the subset of the cusp identified with $[u,K]\times T^{n-1}$. 
For $s\in[0, \infty)$, we denote by $T_{s}$ the cross section $\{s\}\times T^{n-1}$ equipped with the induced metric. With this notation, we have the following:
\begin{align}\label{pian}
\int_{\Omega_{uK}}(L_{\partial_s}\alpha, \alpha)dv&=\int_{\Omega_{uK}}(\{d, i_{\partial_s}\}\alpha, \alpha)dv=\int_{\Omega_{uK}}(d\circ i_{\partial_s}\alpha, \alpha)dv\\ \notag
&=\int_{T_{K}}|i_{\partial_s}\alpha|^{2}d\sigma -\int_{T_{u}}|i_{\partial_s}\alpha|^{2}d\sigma.
\end{align}
We also have
\begin{align} \label{za}
\int_{\Omega_{uK}}(\nabla_{\partial_s}\alpha, \alpha)dv&=\frac{1}{2}\int_{\Omega_{uK}}\partial_{s}|\alpha|^{2}dv \\ \notag
&=\frac{1}{2}\int_{T_{K}}|\alpha|^{2}d\sigma-\frac{1}{2}\int_{T_u}|\alpha|^{2}d\sigma+\frac{n-1}{2}\int_{\Omega_{sK}}|\alpha|^{2}dv.
\end{align}
Using Equations \eqref{Lie}, \eqref{pian}, and \eqref{za}, we have
\begin{align}
\int_{T_{K}}|i_{\partial_s}\alpha|^{2}d\sigma -\int_{T_{u}}|i_{\partial_s}\alpha|^{2}d\sigma=&\Big(\frac{n-1}{2}-k\Big)\int_{\Omega_{sK}}|\alpha|^{2}dv+\int_{\Omega_{sK}}|i_{\partial_{r}}\alpha|^{2}dv\\ \notag
&+\frac{1}{2}\int_{T_{K}}|\alpha|^{2}d\sigma -\frac{1}{2}\int_{T_u}|\alpha|^{2}d\sigma.
\end{align}
Use Definition \eqref{mu} to rewrite this as
\begin{align}\label{pricecusp}
\int_{T_u}\Big(\frac{1}{2}-\mu_\alpha(s)\Big)|\alpha|^{2}d\sigma-\int_{T_K}\Big(\frac{1}{2}-\mu_{\alpha}(K)\Big)|\alpha|^{2}d\sigma=\int_{\Omega_{uK}}\Big(\frac{n-1}{2}-k+\mu_{\alpha}\Big)|\alpha|^{2}dv.
\end{align}
 Specialize now to the case $k<\frac{n-1}{2}$, which makes the right hand side of \eqref{pricecusp} nonnegative.

Since $\alpha$ is $L^{2}$,
\[
\int^{\infty}_{0}\Big(\int_{T_s}|\alpha|^{2}d\sigma \Big)ds=\int_{\Omega_{0\infty}}|\alpha|^{2}dv<\infty.
\]
Hence there exists a sequence $\{K_{i}\}\subset (0,\infty)$ going to $\infty$, such that
\[
\lim_{i\rightarrow\infty}\int_{T_{K_{i}}}|\alpha|^{2}d\sigma  = 0,
\]
and since $\mu_{\alpha}(r)$ is  bounded, we conclude that
\[
\lim_{i\rightarrow\infty}\int_{T_{K_{i}}}\Big(\frac{1}{2}-\mu_{\alpha}(K_{i})\Big)|\alpha|^{2}d\sigma = 0.
\]
Thus, for any $u\in[0, \infty)$ we have 
\begin{align}\label{start}
\int_{T_{u}}\Big(\frac{1}{2}-\mu_{\alpha}(s)\Big)|\alpha|^{2}d\sigma =\int_{\Omega_{u\infty}}\Big(\frac{n-1}{2}-k+\mu_{\alpha}\Big)|\alpha|^{2}dv.
\end{align}

Since $k<\frac{n-1}{2}$ and $\mu_{\alpha}\geq 0$ by construction, Equation \ref{start} gives us the following useful corollary.

\begin{corollary}\label{1st}
	For $k<\frac{n-1}{2}$ and for any $u\in[0, \infty)$, we have
	\begin{align}\label{monotone}
	\int_{T_{u}}\Big(\frac{1}{2}-\mu_{\alpha}(u)\Big)|\alpha|^{2}d\sigma \geq 0.
	\end{align}
	In particular,  $0\leq \mu_{\alpha}(u)<\frac{1}{2}$. Moreover, $\int_{T_{u}}\Big(\frac{1}{2}-\mu_{\alpha}(u)\Big)|\alpha|^{2}d\sigma$ is decreasing.
\end{corollary}

Corollary \ref{1st} and Equation \eqref{start} enables us to estimate the decay of  $\int_{\Omega_{u\infty}}|\alpha|^{2}dv$ as $u$ increases. Define
\[
g(s):= \int_{\Omega_{s\infty}}\Big(\frac{n-1}{2}-k+\mu_{\alpha}\Big)|\alpha|^{2}dv.
\]
For some differentiable $\phi(u)$ to be determined, multiply Equation \eqref{start} by $\phi^\prime(u)$ and integrate from $\sigma$ to $\tau$ to obtain:
\begin{align}\label{CuspPriceR}
\int_{\Omega_{\sigma\tau}}\phi^\prime(s)\Big(\frac{1}{2}-\mu_{\alpha}\Big)|\alpha|^{2}dv&=\phi(\tau)g(\tau)-\phi(\sigma)g(\sigma)-\int^{\tau}_{\sigma}\phi(s)g^\prime(s)ds\\ \notag
&=\phi(\tau)g(\tau)-\phi(\sigma)g(\sigma)+\int_{\Omega_{\sigma\tau}}\phi(s)\Big(\frac{n-1}{2}-k+\mu_{\alpha}\Big)|\alpha|^{2}dv .
\end{align}
Choose 
\begin{align}\label{phi}
\phi(s):= e^{\int^{s}_{0}\frac{[\frac{(n-1)}{2} - k +\mu_{\alpha}(u)]}{1/2 - \mu_{\alpha}(u)}du}\geq e^{(n-1  - 2k)s}.
\end{align}
With this choice,  the two volume integrals in Equation \eqref{CuspPriceR} cancel, reducing  Equation \ref{CuspPriceR} to
\[
\phi(\sigma)g(\sigma)=\phi(\tau)g(\tau),
\]
which we expand as 
\begin{align}\label{PriceR}
\phi(\sigma)\int_{\Omega_{\sigma\infty}}\Big(\frac{n-1}{2}-k+\mu_{\alpha}\Big)|\alpha|^{2}dv=\phi(\tau)\int_{\Omega_{\tau\infty}}\Big(\frac{n-1}{2}-k+\mu_{\alpha}\Big)|\alpha|^{2}dv.
\end{align}

Setting
\[
C_{\alpha}:=\int_{\Omega_{0\infty}}\Big(\frac{n-1}{2}-k+\mu_{\alpha}\Big)|\alpha|^{2}dv,
\]
gives for any $K>0$, 
\[
\int_{\Omega_{K\infty}}\Big(\frac{n-1}{2}-k+\mu_{\alpha}\Big)|\alpha|^{2}dv=\frac{C_{\alpha}}{\phi(K)}\leq C_{\alpha}e^{-(n-1  - 2k)s}.
\]
 
We summarize this discussion with a proposition which is the cusp analog of the Price inequality for geodesic balls proved in \cite{DS17}; so,  it seems natural to refer to the monotonicity inequality in Proposition \ref{l2decay} as a \emph{cuspidal} Price inequality. 

\begin{proposition}\label{l2decay}
Let $\alpha$ be an $L^{2}$-harmonic $k$-form, $k<\frac{n-1}{2}$, on an $n$-dimensional real-hyperbolic manifold. For any $s>0$, on any cusp we have 
\begin{align}\label{decayreal}
\int_{\Omega_{s\infty}}|\alpha|^{2}dv\leq c_{n, k}e^{-(n-1-2k)s}\int_{\Omega_{0\infty}}|\alpha|^{2}dv,
\end{align}
where $c_{n, k}:= \frac{n+1-2k}{n-1-2k}$.
\end{proposition}

\begin{proof}
The argument detailed above proves this proposition for $L^2$-harmonic $k$-forms, $k<(n-1)/2$, on a real-hyperbolic cusp with torus cross section. For a general real-hyperbolic cusp $C:=[0, \infty)\times F$, write $F=\Lambda\backslash \IR^{n-1} $, with $\Lambda$ crystallographic, and let $\Lambda^\prime\leq \Lambda$ be a finite index abelian subgroup. The existence of such a subgroup follows from Bieberbach's theorem, see for example Chapter 3 in \cite{Wolf}. Let $T^{n-1}:=\Lambda'\backslash\IR^{n-1}$ and consider the associated regular Riemannian cover
\[
p^\prime: T^{n-1}\times [0, \infty)\longrightarrow F\times [0, \infty).
\]
Given $\alpha$ on $C$, consider its pull-back ${p^{\prime}}^*\alpha$ to $T^{n-1}\times [0, \infty)$. We conclude using the multiplicativity under covers of both sides in the inequality in \eqref{decayreal}.
\end{proof}

\section{$L^2$-Harmonic Forms on Complex-Hyperbolic Cusps}

Up to finite cover, cusps of finite volume $n$-dimensional complex-hyperbolic manifold are diffeomorphic to $[0,\infty)\times N$, where 
$N$ is a circle bundle with connection over a $2n-1$ torus (\emph{cf}. Remark \ref{hummel}). In other words, the cross-section $N$ is a \emph{nilmanifold}. The metric on the product is given by 
\begin{align}
g= ds^2 + e^{-4s}g_F+ e^{-2s}g_H,
\end{align}
where $g_F$ is a metric on the circle fiber extended to $N$ by the connection, and $g_H$ is the pull-back to $N$ of a flat metric on the $T^{2n-1}$ base. Let $N_s:= \{s\}\times N$ with the induced metric. Let $\Omega_{uK}$ denote the subset of the cusp corresponding to $[u,K]\times N$. Let $\{\frac{\p}{\p s}\}\cup \{e_j\}_{j=1}^{2n-1} $ and $\{ds\}\cup \{\omega^j\}_{j=1}^{2n-1} $ be a local orthonormal frame and dual frame with $e_1$ tangent to the circle fiber. Specializing \eqref{lie} to $\rho = s$ yields the following expression for $L_{\frac{\p}{\p s}}$ acting on $k-$forms.  
\begin{align}L_{\frac{\p}{\p s}} &= \nabla_{\frac{\p}{\p s}}  -\sum_{j=2}^{2n-1}e(\omega^j)e^*(\omega^j)-2e(\omega^1)e^*(\omega^1)\nonumber\\
&= \nabla_{\frac{\p}{\p s}}  -k\cdot Id- e(\omega^1)e^*(\omega^1)+e(ds)e^*(ds),
\end{align}
 where  $Id$  denotes the identity operator on $k$-forms.  Let $J$ denote the complex structure operator. Then $Jds = \pm \omega^1$. 
Let $\alpha$ be a strongly harmonic $k$-form. Then $Jh$ is also a strongly harmonic $k$-form, and defining $\mu_\alpha$ as in \eqref{mu}, we have the equality 
\begin{align}
\mu_{J\alpha}(s) = \frac{\int_{N_s}|i_{e_1}\alpha|^2d\sigma }{\int_{N_s}|\alpha|^2|d\sigma}.
\end{align}
Arguing as in the preceding section, we deduce 
\begin{align}\label{halfprice}
&\int_{\Omega_{u\infty}}(n-k+\mu_{\alpha}-\mu_{J\alpha})|\alpha|^2dv
 = \int_{N_u}\Big(\frac{1}{2}-\mu_{\alpha}\Big)|\alpha|^2d\sigma
\end{align}
Summing \eqref{halfprice} with the corresponding equation for $J\alpha$ yields
\begin{align}\label{Jprice}
&\int_{\Omega_{u\infty}}(2n  -2k )|\alpha|^2dv
 = \int_{N_u}(1-\mu_{\alpha}-\mu_{J\alpha})|\alpha|^2d\sigma.
\end{align}
Arguing as in the proof of Proposition \ref{l2decay} yields the following \emph{cuspidal} Price inequality for complex hyperbolic spaces.
\begin{proposition}\label{jl2decay}
	Let $\alpha$ be a strongly $L^{2}$ harmonic $k$-form, $k<n$, on an $n$-dimensional complex-hyperbolic space.  For any $s>0$, on each cusp, we have 
\begin{align}\label{decaycomplex}
\int_{\Omega_{s\infty}}|\alpha|^{2}dv\leq e^{-2(n-k)s}\int_{\Omega_{0\infty}}|\alpha|^{2}dv.
\end{align} 

\end{proposition}
\begin{proof}
Proceed similarly to Equation \ref{CuspPriceR} in the proof of Proposition \ref{l2decay}, and in this case define for $s>0$
\[
g(s):=\int_{\Omega_{s\infty}}2(n-k)|\alpha|^2dv.
\]	
Thus, following the argument detailed in Section \ref{standard}, this proposition is proved for complex-hyperbolic cusps with nilmanifolds cross sections. (Unlike in Proposition \ref{l2decay}, in this case the constant $c_{n,k}$ can be taken to be  one.) For general cusps with infranilmanifolds cross sections, we conclude using the multiplicativity under covers of both sides in the inequality in \eqref{decaycomplex}.
\end{proof}

\begin{remark}\label{hummel}
	As with real-hyperbolic manifolds, one can always pass to a finite cover of a  complete finite volume complex-hyperbolic manifold for which all cusps have nilmanifold cross sections, see for example \cite{Hummel}. In the arithmetic case, it suffices to pass to a suitable congruence subgroup.  This  fact, however,  is not needed for our estimates.
\end{remark}

\section{$L^2$-Betti Number Estimates away from Middle Degree}\label{ccusps}

We continue to assume $\Gamma$ is always discrete and torsion free. Given  $ \Gamma$ of co-finite volume, $b_{2}^k(\Gamma\backslash\textbf{H}^{n}_{\IC})$ is \emph{always} finite, see Section \ref{introduction} for more details and references. In this section, we  derive effective estimates for $b_{2}^k$ on such ball quotients. We start with a proposition. 

\begin{proposition}\label{ave}
Let $\Gamma$ be co-finite volume. Decompose $\Gamma\backslash\textbf{H}^{n}_{\IC}$ as a disjoint union
\[ 
\Gamma\backslash\textbf{H}^{n}_{\IC} = M_0\cup (\cup_jC_j), 
\] 
where each $C_j$ is a cusp, parameterized by $[0,\infty)\times N^j$. Let $\Omega_{uK}^j$ denote the subset of $C_j$ parameterized by $[u,K)\times N^j$. Suppose that: 
\begin{align}\label{ass1} 
\int_{\pi^{-1}(M_{0})}\binom{2n}{k}\int_{S(\mathcal{H}^k(\Gamma\backslash\textbf{H}^{n}_{\IC}))}|h(\bar{p})|^2d\sigma_hd\bar{v}_{\bar{p}} <\frac{Vol(M_{0})}{2Vol(\Gamma\backslash\textbf{H}^{n}_{\IC})}.
\end{align}
Then for $k<n$ the average value of 
\[ \quad  \binom{2n}{k}\int_{S(\mathcal{H}^k(\Gamma\backslash\textbf{H}^{n}_{\IC}))}  |h(\bar{p})|^2d\sigma_h \quad \text{on} \quad \pi^{-1}(\cup_j\Omega_{01}^j)
\] 
is greater than
\[
\frac{1}{4 \sum_j Vol(\Omega_{01}^j)}.
\]
\end{proposition}

\begin{proof} 
Combining \eqref{ass1} and \eqref{zero} with Lemma \ref{split}, we obtain 
\begin{align}\label{input}
\sum_j\int_{\Omega_{0\infty}^j}\int_{S(\mathcal{H}^k(\Gamma\backslash\textbf{H}^{n}_{\IC}))}|h|^2(p)d\sigma_hdv_p\geq \frac{1}{2}.
\end{align}
Using the Price inequality given in Proposition \ref{jl2decay}, for each $j$ we have: 
\begin{align}\notag
\int_{S(\mathcal{H}^k(\Gamma\backslash\textbf{H}^{n}_{\IC})))}\int_{\Omega_{1\infty}^j}|h|^2(p)dv_pd\sigma_h\leq  e^{-2(n-k)}\int_{S(\mathcal{H}^k(\Gamma\backslash\textbf{H}^{n}_{\IC})))}\int_{\Omega_{0\infty}^j}|h|^2(p)dv_pd\sigma_h,
\end{align}
so that
\begin{align}\notag
\int_{S(\mathcal{H}^k(\Gamma\backslash\textbf{H}^{n}_{\IC})))}\int_{\Omega_{01}^j}|h|^2(p)dv_pd\sigma_h\geq  (1-e^{-2(n-k)})\int_{S(\mathcal{H}^k(\Gamma\backslash\textbf{H}^{n}_{\IC})))}\int_{\Omega_{0\infty}^j}|h|^2(p)dv_pd\sigma_h.
\end{align}
Combining this last estimate with \eqref{input} yields
\begin{align}\notag
\sum_j\int_{\Omega_{01}^j}\int_{S(\mathcal{H}^k(\Gamma\backslash\textbf{H}^{n}_{\IC}))}|h(p)|^2d\sigma_hdv_p\geq 
\frac{1-e^{-2(n-k)}}{2}>\frac{1}{4},
\end{align}
and the result follows.
\end{proof}

We can now estimate from above the dimension of the space of $L^2$-harmonic forms on a complete finite volume  complex-hyperbolic manifold.

\begin{corollary}\label{L2C}
Let $\Gamma $ be co-finite volume. Decompose $\Gamma\backslash\textbf{H}^{n}_{\IC}$ as a disjoint union
\[ 
\Gamma\backslash\textbf{H}^{n}_{\IC} = M_0\cup (\cup_jC_j), 
\] 
where each $C_j$ is a cusp parameterized by $[0,\infty)\times N^j$. Let $\Omega_{uK}^j$ denote the subset of $C_j$ parameterized by $[u,K)\times N^j$.
Then for some $c_{n,k}, \tilde c_{n,k}>0$, 
\begin{align}\label{411}
b_{2}^k(\Gamma\backslash\textbf{H}^{n}_{\IC}) \leq \tilde c_{n, k}\big[Vol(\Gamma\backslash\textbf{H}^{n}_{\IC})V_{min}(M_0)^{\frac{k-n}{n}}+ Vol(\cup_jC_j)V_{min}(\cup_j\Omega_{01}^j)^{\frac{k-n}{n}}\big],
\end{align}
and
\begin{align}\label{pain}
b_{2}^k(\Gamma\backslash\textbf{H}^{n}_{\IC}) \leq  c_{n, k} Vol(\Gamma\backslash\textbf{H}^{n}_{\IC})V_{min}(M_0)^{\frac{k-n}{n}} .
\end{align}
\end{corollary}

\begin{proof}
First, recall that by Corollary \ref{linearh}, if $\bar{p}\in S(\Lambda^k T(\Gamma\backslash\textbf{H}^{n}_{\IC}))$ is such that $h(\bar{p})\neq 0$ for some $h\in\mathcal{H}^k(\Gamma\backslash\textbf{H}^{n}_{\IC})$ then:
\[
b^k_{2}(\Gamma\backslash\textbf{H}^{n}_{\IC})=\frac{\max\{| h(\bar{p})|^2: h\in S(\mathcal{H}^k(\Gamma\backslash\textbf{H}^{n}_{\IC}))\}}{\int_{S(\mathcal{H}^k(\Gamma\backslash\textbf{H}^{n}_{\IC}))}|h(\bar{p})|^2d\sigma_h}.
\]
If there exists a point $\bar{p}\in\pi^{-1}(M_{0})$, such that
\[
\binom{2n}{k}\int_{S(\mathcal{H}^k(\Gamma\backslash\textbf{H}^{n}_{\IC}))}|h(\bar{p})|^2d\sigma_h\geq \frac{1}{2Vol(\Gamma\backslash\textbf{H}^{n}_{\IC})},
\]
then
\[
b^k_2(\Gamma\backslash\textbf{H}^{n}_{\IC})\leq \max\{| h(\bar{p})|^2: h\in S(\mathcal{H}^k(\Gamma\backslash\textbf{H}^{n}_{\IC}))\}\binom{n}{k}\cdot 2Vol(\Gamma\backslash\textbf{H}^{n}_{\IC}),
\]
and the  Price inequality given in Lemma \ref{PriceC} combined with Moser iteration (\emph{cf}. Lemma 41 and Lemma 51 in \cite{DS17}) imply  that there exists a constant $d_{n, k}>0$ such that
\[
b^k_2(\Gamma\backslash\textbf{H}^{n}_{\IC})\leq d_{n, k}Vol(\Gamma\backslash\textbf{H}^{n}_{\IC})V_{min}(M_0)^{\frac{k-n}{n}}.
\]
Suppose now that for every $\bar{p}\in \pi^{-1}(M_0)$ we have 
\[
\binom{2n}{k}\int_{S(\mathcal{H}^k(\Gamma\backslash\textbf{H}^{n}_{\IC}))}|h(\bar{p})|^2d\sigma_h< \frac{1}{2Vol(\Gamma\backslash\textbf{H}^{n}_{\IC})}.
\]
This implies that 
\[
\binom{2n}{k}\int_{\pi^{-1}(M_{0})}\int_{S(\mathcal{H}^k(\Gamma\backslash\textbf{H}^{n}_{\IC}))}|h(p)|^2d\sigma_h d\bar{v}_{\bar{p}}<\frac{Vol(M_{0})}{2Vol(\Gamma\backslash\textbf{H}^{n}_{\IC})}.
\]
 Proposition \ref{ave} then implies the existence of $\bar{p}\in \pi^{-1}(\cup_j\Omega_{01}^j)$ with
\[
\binom{2n}{k}\int_{S(\mathcal{H}^k(\Gamma\backslash\textbf{H}^{n}_{\IC}))}|h(\bar{p})|^2d\sigma_h\geq \frac{1}{4 \sum_j Vol(\Omega_{01}^j)}.
\]
Thus 
\begin{align}
b^k_{2}(\Gamma\backslash\textbf{H}^{n}_{\IC})\leq \max\{| h(\bar p)|^2: h\in S(\mathcal{H}^k(\Gamma\backslash\textbf{H}^{n}_{\IC}))\}\binom{2n}{k}\cdot 4 \sum_j Vol(\Omega_{01}^j).
\end{align}
Using the Price inequality (Lemma \ref{PriceC}) and Moser iteration (Lemma 45 and Lemma 51 in \cite{DS17}),  we then have for some constant $e_{n, k}>0$
\[
b^k_2(\Gamma\backslash\textbf{H}^{n}_{\IC})\leq e_{n, k}V_{min}(\cup_j\Omega_{01}^j)^{\frac{k-n}{n}}\Big(4 \sum_j Vol(\Omega_{01}^j)\Big).
\]
Setting $c_{n, k}:=\max\{d_{n,k},e_{n,k}\}$ completes the proof of \eqref{411}. To deduce \eqref{pain} we note that    $\frac{V_{min}(M_0)}{V_{min}(\cup_j\Omega^j_{01})}$ is uniformly bounded, and \eqref{pain} follows from \eqref{411}.
\end{proof}

Similarly, we have the following proposition and corollary. \\

\begin{proposition}\label{aveR}
	Let $ \Gamma $ be co-finite volume. Decompose $\Gamma\backslash\textbf{H}^{n}_{\IR}$  as a disjoint union
	\[ 
	\Gamma\backslash\textbf{H}^{n}_{\IR} = M_0\cup (\cup_jC_j), 
	\] 
	where each $C_j$ is a cusp, parameterized by $[0,\infty)\times N^j$. Let $\Omega_{uK}^j$ denote the subset of $C_j$ parameterized by $[u,K)\times N^j$. Suppose that: 
	\begin{align}\label{ass1R}\notag
	\int_{\pi^{-1}(M_{0})}\binom{n}{k}\int_{S(\mathcal{H}^k(\Gamma\backslash\textbf{H}^{n}_{\IR}))}|h(\bar{p})|^2d\sigma_hd\bar{v}_{\bar{p}} <\frac{Vol(M_{0})}{2Vol(\Gamma\backslash\textbf{H}^{n}_{\IR})}.
	\end{align}
	For $k<(n-1)/2$ set 
	\[
	R_{n, k}:= \frac{1}{n-1-2k}\ln{\frac{2(n+1-2k)}{n-1-2k}}.
	\]
	Then the average value $\binom{n}{k}\int_{S(\mathcal{H}^k(\Gamma\backslash\textbf{H}^{n}_{\IR}))}  |h(\bar{p})|^2d\sigma_h$ on $\pi^{-1}(\cup_j\Omega_{0R_{n, k}}^j)$ is greater than
	\[
	\frac{1}{4 \sum_j Vol(\Omega_{0R_{n, k}}^j)}.
	\]
\end{proposition}

\begin{proof}
	Combining the hypotheses with \eqref{zero} and Lemma \ref{split}, we obtain 
	\begin{align}\label{inputr}
	\sum_j\int_{\Omega_{0\infty}^j}\int_{S(\mathcal{H}^k(\Gamma\backslash\textbf{H}^{n}_{\IR}))}|h|^2(p)d\sigma_hdv_p\geq \frac{1}{2}.
	\end{align}
	Using the Price inequality on real hyperbolic cusps proved in Proposition \ref{l2decay}, for each $j$ we have: 
	\begin{align}\notag
	\int_{S(\mathcal{H}^k(\Gamma\backslash\textbf{H}^{n}_{\IR})))}\int_{\Omega_{R_{n, k}\infty}^j}|h|^2(p)dv_pd\sigma_h\leq  \frac{1}{2}\int_{S(\mathcal{H}^k(\Gamma\backslash\textbf{H}^{n}_{\IR})))}\int_{\Omega_{0\infty}^j}|h|^2(p)dv_pd\sigma_h,
	\end{align}
	so that
	\begin{align}\notag
	\int_{S(\mathcal{H}^k(\Gamma\backslash\textbf{H}^{n}_{\IR})))}\int_{\Omega_{0R_{n, k}}^j}|h|^2(p)dv_pd\sigma_h\geq  \frac{1}{2}\int_{S(\mathcal{H}^k(\Gamma\backslash\textbf{H}^{n}_{\IR})))}\int_{\Omega_{0\infty}^j}|h|^2(p)dv_pd\sigma_h.
	\end{align}
	Combining this last estimate with \eqref{inputr} yields
	\begin{align}\notag
	\sum_j\int_{\Omega_{0R_{n, k}}^j}\int_{S(\mathcal{H}^k(\Gamma\backslash\textbf{H}^{n}_{\IR}))}|h(p)|^2d\sigma_hdv_p>\frac{1}{4},
	\end{align}
	and the result follows.
\end{proof}

We can now estimate outside the critical degree the dimension of the space of $L^2$-harmonic forms on a complete finite volume  real-hyperbolic manifold.

\begin{corollary}\label{rhypcusps}
	Let $ \Gamma $ be co-finite volume. Decompose $\Gamma\backslash\textbf{H}^{n}_{\IR}$ as a disjoint union
	\[ 
	\Gamma\backslash\textbf{H}^{n}_{\IR} = M_0\cup (\cup_jC_j), 
	\] 
	where each $C_j$ is a cusp parameterized by $[0,\infty)\times T^{n-1}_j$. Let $\Omega_{uK}^j$ denote the subset of $C_j$ parameterized by $[u,K)\times T^{n-1}_j$. For any $k<(n-1)/2$, there exists some $\tilde a_{n,k}, a_{n,k}>0$ such that
	\begin{align}\label{Kate}
	b_{2}^k(\Gamma\backslash\textbf{H}^{n}_{\IR}) \leq \tilde a_{n, k}\big[Vol(\Gamma\backslash\textbf{H}^{n}_{\IR})V_{min}(M_0)^{\frac{2k+1-n}{n-1}}+ Vol(\cup_jC_j)V_{min}(\cup_j\Omega_{0R_{n, k}}^j)^{\frac{2k+1-n}{n-1}}\big],
	\end{align}
and 
\begin{align}\label{Luca} 
	b_{2}^k(\Gamma\backslash\textbf{H}^{n}_{\IR}) \leq   a_{n, k} Vol(\Gamma\backslash\textbf{H}^{n}_{\IR})V_{min}(M_0)^{\frac{2k+1-n}{n-1}}.
	\end{align}
\end{corollary}

\begin{proof}
	As in the proof of Corollary \ref{L2C}, if there exists a point  $\bar{p}\in\pi^{-1}(M_{0})$ such that
	\[
	\binom{n}{k}\int_{S(\mathcal{H}^k(\Gamma\backslash\textbf{H}^{n}_{\IR}))}|h(\bar{p})|^2d\sigma_h\geq \frac{1}{2Vol(\Gamma\backslash\textbf{H}^{n}_{\IR})}.
	\]
    Combining the Price inequality for real-hyperbolic spaces (Theorem 87 in \cite{DS17}) and Moser iteration (Lemma 45 and Lemma 51 in \cite{DS17}),  we obtain that there exists a constant $d_{n, k}>0$ such that
	\[
	b^k_2(\Gamma\backslash\textbf{H}^{n}_{\IR})\leq d_{n, k}Vol(\Gamma\backslash\textbf{H}^{n}_{\IR})V_{min}(M_0)^{\frac{2k+1-n}{n-1}}.
	\]
	Suppose now that for any $\bar{p}\in \pi^{-1}(M_0)$ we have 
	\[
	\binom{n}{k}\int_{S(\mathcal{H}^k(\Gamma\backslash\textbf{H}^{n}_{\IR}))}|h(\bar{p})|^2d\sigma_h< \frac{1}{2Vol(\Gamma\backslash\textbf{H}^{n}_{\IR})}.
	\]
	By Proposition \ref{ave}, there exists $\bar{p}\in \pi^{-1}(\cup_j\Omega_{0R_{n, k}}^j)$ with
	\[
	\binom{n}{k}\int_{S(\mathcal{H}^k(\Gamma\backslash\textbf{H}^{n}_{\IR}))}|h(\bar{p})|^2d\sigma_h\geq \frac{1}{4 \sum_j Vol(\Omega_{0R_{n, k}}^j)}.
	\]
	By the above Price and Moser inequalities, for some constant $e_{n, k}>0$,
	\[
	b^k_2(\Gamma\backslash\textbf{H}^{n}_{\IR})\leq e_{n, k}V_{min}(\cup_j\Omega_{0R_{n, k}}^j)^{\frac{2k+1-n}{n-1}}\Big(4 \sum_j Vol(\Omega_{0R_{n, k}}^j)\Big).
	\]
	The proof of \eqref{Kate} follows upon setting $\tilde a_{n, k}:=\max\{d_{n,k},e_{n,k}\}$. The proof that \eqref{Luca} follows from \eqref{Kate} follows as in Corollary \ref{L2C} after noting $R_{n.k}$ is uniformly bounded. 
\end{proof}


\section{Critical Degree for Real-Hyperbolic Manifolds} 

Corollary \ref{rhypcusps} does not cover the case $b_{2}^k(\Gamma\backslash\textbf{H}^{n}_{\IR})$ when $n=2k+1$. 
For this critical degree case, Equation \eqref{start} reduces to  
\begin{align}\label{monotone2}
\int_{T_{s}^a}\Big(\frac{1}{2}-\mu_\alpha(s)\Big)|\alpha|^{2}d\sigma=\int_{\Omega_{s\infty}^a}\mu_\alpha|\alpha|^{2}dv>0,
\end{align}
which again implies $0\leq \mu_\alpha(s)<\frac{1}{2}$. Unfortunately, this identity is difficult to use without having control of $\mu_\alpha(s)$. Thus, we Fourier expand to obtain additional information and control on the size of an $L^2$-harmonic form of critical degree on a hyperbolic cusp.

\subsection{Fourier Primitives}

Consider a standard hyperbolic cusp, say $C_a$, of a complete finite volume real-hyperbolic manifold. Recall that a standard real-hyperbolic cusp is isometric to $[0,\infty)\times T^a$, where $T^a=\Lambda_a\backslash \IR^{n-1}$ is an $(n-1)$-dimensional real torus associated to a full rank lattice of translations $\Lambda_a$ acting on $\IR^{n-1}$. Moreover, $C_a$ is naturally equipped with the metric $g_{-1}=ds^2+e^{-2s}g_{T^a}$. Let $\{t^k\}_{k=1}^{n-1}$ be Euclidean coordinates on $\IR^{n-1}$. Recall also that given a complete finite volume real-hyperbolic manifold, one can always pass to a finite cover where all of the cusps are standard, see for example \cite{Hummel}. Now, given a strongly harmonic (i.e., in the kernel of $d$ and $d^*$) $L^2$-form $h$ of degree $k$ on $C_{a}$, we Fourier expand 
\begin{align}
h = \sum_{v\in\Lambda^*_a}h^v ,
\end{align}
where for any $k$
$$L_{\frac{\p}{\p t^k}}h^v= 2\pi iv_kh^v.$$
For $v\not = 0$, set 
\[
b^v := \sum_k \frac{v_k}{2\pi i|v|^2}i_{\frac{\p}{\p t^k}}h^v.
\]
Then
\begin{align}\label{bv}
db^v = \sum_k \frac{v_k}{2\pi i|v|^2}L_{ \frac{\p}{\p t^k}}h^v = h^v.
\end{align}
For any $R_1>R_2>0$, by   \eqref{bv}  we have 
\begin{align}\label{bv1}
\int_{\Omega_{R_1R_2}^a}|h|^2 dv &=  \int_{\Omega_{R_1R_2}^a}\sum_{v\in \Lambda^*_a\setminus\{0\}}\langle  db^v,h^v\rangle dv + \int_{\Omega_{R_1R_2}^a}|h^0|^2dv\nonumber\\ \notag
&=  \int_{T_{R_2}^a}\sum_{v\in \Lambda^*_a\setminus\{0\}}\langle  ds\wedge b^v ,h^v\rangle dv-\int_{T_{R_1}^a}\sum_{v\in \Lambda^*_a\setminus\{0\}}\langle   ds\wedge b^v  ,h^v\rangle dv \\ 
& + \int_{\Omega_{R_1R_2}^a}|h^0|^2dv,
\end{align}
where in the second equality we used the fact that $d^{*}h^v=0$ for all $v\in\Lambda^*_a$. Since $h\in L^2$, taking the limit as $R_2\to \infty$ in Equation \eqref{bv1} yields the estimate
\begin{align}\label{libel}
\int_{\Omega_{R_1 \infty}^a}|h-h^0|^2 dv  
&=  -\int_{T_{R_1}^a}\sum_{v\in \Lambda^*_a\setminus\{0\}}\langle   ds\wedge b^v  ,h^v\rangle dv  \nonumber\\
&\leq  \int_{T_{R_1}^a}\sum_{v\in \Lambda^*_a\setminus\{0\}}  \frac{e^{-R_1}}{2\pi |v| } |h^v|^2   dv  .
\end{align}
Setting
\[
\boxed{\delta_{\Lambda_a}:= \inf_{v\in\Lambda^*_a\setminus\{0\}}|v|,}
\]
we have for any $R>0$:
\begin{align}\label{nonzeromode}
e^{R}\int_{\Omega_{R\infty}^a}|h-h^0|^2 dv   &\leq  \frac{1}{2\pi \delta_{\Lambda_a} }\int_{T_{R }^a}   |h -h^0|^2   d\sigma.
\end{align}
Integrating \eqref{nonzeromode} from $R_0$ to $\infty$ yields 
\begin{align}\label{partest}
\int_{\Omega_{R_0 \infty}^a}e^s|h-h^0|^2 dv   &\leq \Big(e^{R_{0}}+ \frac{1}{2\pi \delta_{\Lambda} }\Big)\int_{\Omega_{R_0 \infty}^a}  |h -h^0|^2  d\sigma.
\end{align}
Next, we need to give a bound on the zero mode $h^0$. It has the form 
\[
h^0 = h_{0I}^0(s) dt^I+h_{1J}^0(s) ds\wedge dt^J,
\]
where we use multi-index notation $dt^I = dt^{i_1}\wedge\cdots\wedge dt^{i_{|I|}}.$ Since $dh^0 = 0$,  $h_{0I}^0(s)$ is constant. 
Next, let $*$ be the Hodge operator associated to the hyperbolic metric acting on $k=(n-1)/2$-forms. We can write
\begin{align}\notag
\ast h^0 = \pm  h_{0I}^0  ds\wedge dt^{I^c}\pm e^{-2s}h_{1J}^0(s) dt^{J^c},
\end{align}
where by $I^{c}$ and $J^c$ we indicate the multi-index complement in $\IR^{n-1}$. Since $d^*h^0=\pm (\ast d\ast) h^0 = 0$, we have that $d(\ast h^{0})=0$. This implies $e^{-2s}h^{0}_{1J}(s)$ is constant so that 
\[
h^{0}_{1J}(s)=h^{0}_{1J}(0)e^{-2s}.
\]
Hence for $k=\frac{n-1}{2}$, we have the equality 
\begin{align}\notag
|h^0(s)|^2 = |h_{0I}^0(0)|^2 e^{ \frac{(n-1)}{2}s}+|h_{1J}^0(0)|^2e^{\frac{(n+1)}{2}s},
\end{align}
and we conclude that
\[
\int_{C_{a}}|h^0|^2dv< \infty \quad \Leftrightarrow \quad h^0_{0I}(0)=0, \quad h^{0}_{1J}(0)=0, \quad \Rightarrow \quad h^{0}=0.
\]
  Equation  \eqref{partest} now simplifies to 
\begin{align}\notag
\int_{\Omega_{R_{0}\infty}^a}e^{s}|h|^2dv\leq\Big(e^{R_{0}}+\frac{1}{2\pi\delta_{\Lambda_a}}\Big)\int_{\Omega_{R_{0}\infty}^a}|h|^2dv.
\end{align} 
We summarize this discussion with a lemma.

\begin{lemma}\label{2nd}
Let $C_a$ be a standard real-hyperbolic cusp of dimension $n$, and let $h \in L^2$ be a strongly harmonic form of degree $(n-1)/2$ on it. For any $R_{0}\in[0, \infty)$, we have the integral inequality:
\begin{align}\label{criticalI}
\int_{\Omega_{R_{0}\infty}^a}e^{s-R_0}|h|^2dv\leq\Big(1+\frac{e^{-R_{0}}}{2\pi\delta_{\Lambda_a} }\Big)\int_{\Omega_{R_{0}\infty}^a}|h|^2dv.
\end{align}
\end{lemma}

\begin{remark}\label{nonstandard}
We remark that Lemma \ref{2nd} holds true for non-standard real-hyperbolic cusps as well. To see this, let $C$ be a real-hyperbolic cusp diffeomorphic to $[0,\infty)\times F$, where $F=\Lambda\backslash\IR^{n-1}$ is an $(n-1)$-dimensional flat manifold associated to a full rank lattice  $\Lambda$ acting on $\IR^{n-1}$. The lattice $\Lambda$ does \emph{not} need to be a lattice of translations in $\IR^{n-1}$, but by Bieberbach's theorem  (\emph{cf}. Chapter 3 in \cite{Wolf}), we can always find a
finite index lattice of translations of minimal index  $\Lambda_a\leq \Lambda$s such that   $[\Lambda: \Lambda_a]\leq C_{n-1}$, where $C_{n-1}$ is a positive constant depending on the dimension only. Consider then the associated Riemannian cover of index $[\Lambda: \Lambda_a]$
\[
p:T^a\times [0, \infty)\longrightarrow F\times [0, \infty),
\]
where $T^a$ is the $(n-1)$-torus associated to the lattice of translations $\Lambda_{a}$. Given an $L^2$ strongly harmonic form $h$ on $C$, $p^{*}h$ is $L^2$ and strongly harmonic on $  [0, \infty)\times T_a$, with the pulled back metric. Apply Lemma \ref{2nd} to $p^{*}h$ and use the mutiplicativity of the integral under pullback to finite Riemannian covers to conclude
\begin{align}\label{Fgeneral}
\int_{\Omega_{R_{0}\infty}}e^{s-R_0}|h|^2dv\leq\Big(1+\frac{e^{-R_{0}}}{2\pi\delta_{\Lambda_a} }\Big)\int_{\Omega_{R_{0}\infty}}|h|^2dv,
\end{align}
where now $\Omega_{R_{0}\infty}=[R_{0}, \infty)\times F$. 
\end{remark}

We now combine the inequality in \eqref{criticalI} and Remark \ref{nonstandard} with the peaking argument presented in Section \ref{previsited} to obtain the following corollary.

\begin{corollary}\label{coroest}
Let	$(\Gamma\backslash\textbf{H}^{n}_{\IR}, g_{\IR})$ be a complete finite volume real-hyperbolic manifold of odd dimension $n=2k+1$ for some integer $k\geq 1$. For $R_{0}\in [0, \infty)$, on any cusp $C_{a}= [0, \infty)\times F$ (not necessarily standard) we have	
\begin{align}\notag
\int_{\Omega_{(R_{0}+1)\infty}^a}&\int_{S(\mathcal{H}^k(\Gamma\backslash\textbf{H}^{n}_{\IR}))} |h|^2d\sigma_hdv  \\ \notag
&\leq\frac{e^{-R_{0}}}{2\pi\delta_{\Lambda_a}}\Big(e-1-\frac{e^{-R_{0}}}{2\pi\delta_{\Lambda_a}}\Big)^{-1} 
\int_{\Omega_{R_0(R_{0} +1)}^a}\int_{S(\mathcal{H}^k(\Gamma\backslash\textbf{H}^{n}_{\IR}))}|h|^2d\sigma_hdv,
\end{align}
where $\Lambda_{a}$ is a lattice of translation of minimal index in $\Lambda$, where $F:=\Lambda\backslash \IR^{n-1}$.
\end{corollary}

\begin{proof}
If the cusp has a torus cross section, by Lemma \ref{2nd} we have
\[
\int_{\Omega^a_{(R_0+1)\infty}}e^{s-R_{0}}|h|^2dv\leq \Big(1+\frac{e^{-R_{0}}}{2\pi\delta_{\Lambda_a} }\Big)\int_{\Omega_{R_{0}\infty}^a}|h|^2dv - \int_{\Omega^a_{R_{0}(R_{0}+1)}}e^{s-R_{0}}|h|^2dv,
\]	
so that
\[
e\int_{\Omega^a_{(R_0+1)\infty}}|h|^2dv\leq \Big(1+\frac{e^{-R_{0}}}{2\pi\delta_{\Lambda_a} }\Big)\Big(\int_{\Omega^a_{R_{0} (R_{0}+1)}}|h|^2dv+\int_{\Omega^a_{(R_{0}+1)\infty}}|h|^2dv\Big) - \int_{\Omega^a_{R_{0} (R_{0}+1)}}|h|^2dv,
\]
which implies
\[
\Big(e-1-\frac{e^{-R_{0}}}{2\pi\delta_{\Lambda_a}}\Big)\int_{\Omega^{a}_{(R_{0}+1) \infty}}|h|^2dv\leq \frac{e^{-R_{0}}}{2\pi\delta_{\Lambda_a}}\int_{\Omega^a_{R_{0} (R_{0}+1)}}|h|^2dv.
\]
We then conclude by integrating over $S(\mathcal{H}^k(\Gamma\backslash\textbf{H}^{n}_{\IR}))$. If the cusp is not standard, let $F=\Lambda\backslash\IR^{n-1}$ be the associated flat cross-section, and let $\Lambda_{a}$ be an abelian subgroup of minimal index in $\Lambda$, and  conclude the argument by appealing to Remark \ref{nonstandard}.
\end{proof}

\begin{remark}
We observe that when dealing with lattices that are close to a dilation of $\IZ^{n-1}$ (as one might expect to find in congruence subgroup cases), the preceding estimate can be used to extend our Betti number estimates to the cusped case. In more general lattices, more care is required balancing the size of $\delta_{\Lambda_a}$, volumes, and local injectivity radii. 
\end{remark}

\subsection{Cusps and Lattices}\label{rlattices}

Let $(M^n:=\Gamma\backslash\textbf{H}^{n}_{\IR}, g_{\IR})$ be a complete finite volume real-hyperbolic manifold with the property that all of its cusps are \emph{standard}. Recall that one can always arrange this to be the case by passing to a finite regular cover, see for example \cite{Hummel}. Also, assume the dimension to be $n=2k+1$ for some integer $k\geq 1$. As in Section \ref{previsited}, we let
\[
\pi: S(\Lambda^k TM)\rightarrow M
\]
be the unit sphere bundle, where the total space is equipped with the metric induced from $g_{\IR}$ and the standard metric on the spherical fiber whose volume is normalized to one. We denote by $d\bar{v}$ the associated Riemannian measure on the total space.

Next, decompose $M^n$ as the disjoint union
\[
M= M_0\cup(\cup_aC_a),
\]
where each $C_a$ is a cusp parametrized by $[0,\infty)\times T^a$.  Set 
\[
M_{k} := M_0\cup(\cup_a([0,k)\times T^a)), \quad C_{a,k} := [k,\infty)\times T^a.
\]
Thus, we have 
\begin{align}\notag
Vol(C_{a,k}) = e^{-k(n-1)}Vol(C_{a,0}).
\end{align}
Let $\nu(M)$ be the smallest integer such that 
\begin{align}\label{lowdef}
\binom{n}{k}\int_{\pi^{-1}(M_{\nu(M)})}\int_{S(\mathcal{H}^k(M))}|h(\bar{p})|^2d\sigma_hd\bar{v}_{\bar{p}} \geq \frac{\sum_{j=1}^{\nu(M)+1}j^{-2}}{2\zeta(2)}.
\end{align}  Here $\zeta(2)$ is the Riemann zeta function evaluated at $2$, introduced to normalize the right hand side of \eqref{lowdef} to be smaller than $\frac{1}{2}$.  Equation \eqref{one} ensures that  $\nu(M)$ is
well-defined. If $\nu(M)>0$, we have
\begin{align}\label{lesshalf}
\binom{n}{k}\int_{\pi^{-1}(M_{\nu(M)-1})}\int_{S(\mathcal{H}^k(M))}|h(\bar{p})|^2d\sigma_hd\bar{v}_{\bar{p}} < \frac{\sum_{j=1}^{\nu(M)}j^{-2}}{2\zeta(2)},
\end{align}
which then implies
\begin{align}\notag
\binom{n}{k}\int_{\pi^{-1}(M_{\nu(M)}\setminus M_{\nu(M)-1})}\int_{S(\mathcal{H}^k(M))}|h(\bar{p})|^2d\sigma_h d\bar{v}_{\bar{p}} \geq \frac{(\nu(M)+1)^{-2}}{2\zeta(2)}.
\end{align}
Hence there exists $\bar{p}\in \pi^{-1}(M_{\nu(M)}\setminus M_{\nu(M)-1})$ such that
\begin{align}\label{lowb}
\binom{n}{k}\int_{S(\mathcal{H}^k(M))}|h(\bar{p})|^2d\sigma_h \geq \frac{(\nu(M)+1)^{-2}}{2\zeta(2)Vol(M_{\nu(M)}\setminus M_{\nu(M)-1})}.
\end{align}
At the cost of decreasing constants on the right hand side of \eqref{lowb}, it is convenient to further restrict the location of $\bar{p}$ satisfying (a modified) \eqref{lowb}. To effect this, define the index set  $I(\nu(M))$  so that 
\[
\boxed{a\in I(\nu(M)) \quad \Longleftrightarrow \quad \frac{e^{2-\nu(M)}}{2\pi \delta_{\lambda_a}}< \frac{1}{4(\nu(M)+1)^2}.}
\]
Next, we claim that
\begin{align}\label{nu}
\binom{n}{k}\int_{\pi^{-1}((M_{\nu(M)}\setminus M_{ \nu(M)-1})\cap (\cup_{a\in I(\nu(M))}C_a))}\int_{S(\mathcal{H}^k(M))}|h(\bar{p})|^2d\sigma_h d\bar{v}_{\bar{p}}< \frac{1}{8(\nu(M)+1)^2}.
\end{align}
To verify this, use Corollary \ref{coroest} with  $R_{0}=\nu(M)-2$  to obtain
\begin{align}\label{r0}
\sum_{a\in I(\nu(M))}&\int_{\pi^{-1}(\Omega^{a}_{(\nu(M)-1)\infty})}\int_{S(\mathcal{H}^k(M))}|h(\bar{p})|^2d\sigma_h d\bar{v}_{\bar{p}}< \\ \notag 
&\frac{1}{4(\nu(M)+1)^2}\sum_{a\in I(\nu(M))}\int_{\pi^{-1}(\Omega^{a}_{(\nu(M)-2)(\nu(M)-1)})}\int_{S(\mathcal{H}^k(M))}|h(\bar{p})|^2d\sigma_h d\bar{v}_{\bar{p}}.
\end{align}
Seeking a contradiction, we note that if Equation \eqref{nu} does not hold,  then Equation \eqref{r0} implies
\[
\sum_{a\in I(\nu(M))}\binom{n}{k}\int_{\pi^{-1}(\Omega^{a}_{(\nu(M)-2)(\nu(M)-1)})}\int_{S(\mathcal{H}^k(M))}|h(\bar{p})|^2d\sigma_h d\bar{v}_{\bar{p}}\geq \frac{4(\nu(M)+1)^2}{8(\nu(M)+1)^2}=\frac{1}{2}.
\] 
This contradicts Equation \eqref{lesshalf}, and we have verified that \eqref{nu} holds. Hence, by Equation \eqref{one} we also know that
\begin{align}\notag
\binom{n}{k}\int_{\pi^{-1}((M_{\nu(M)}\setminus M_{ \nu(M)-1})\cap ((\cup_{a\in I(\nu(M))}C_a)^c)}\int_{S(\mathcal{H}^k(M))}|h(\bar{p})|^2d\sigma_h d\bar{v}_{\bar{p}}\geq \frac{1}{8(\nu(M)+1)^2},
\end{align}
and consequently there exists $\bar p\in \pi^{-1}(M_{\nu(M)}\setminus M_{\nu(M)-1})\cap (\cup_{a\in I(\nu(M))}C_a)^c)$ such that
\begin{align}\label{lowb2}
\binom{n}{k}\int_{S(\mathcal{H}^k(M))}|h(\bar{p})|^2d\sigma_h \geq \frac{(\nu(M)+1)^{-2}}{8Vol(M_{\nu(M)}\setminus M_{\nu(M)-1})}.
\end{align}
Applying Corollary \ref{linearh} yields
\begin{align}\label{fund}\notag
b^k_2(M) & = \frac{\max_{h\in S^{b-1}}|h(\bar{p})|^2 }{\int_{S^{b-1}}|h(\bar{p})|^2d\sigma_h}\\ \notag
&\leq  8\binom{n}{k}(\nu(M)+1)^2Vol(M_{ \nu(M)}\setminus M_{\nu(M)-1})\max_{h\in S(\mathcal{H}^k(M))}|h(\bar{p})|^2 \\  
&\leq 8\binom{n}{k}e^{-(n-1)\nu(M)}(\nu(M)+1)^2Vol(M_{1}\setminus M_{0})\max_{h\in S(\mathcal{H}^k(M))}|h(\bar{p})|^2,
\end{align}
where in the last inequality we used the explicit form of the hyperbolic metric on cusps. Now, the injectivity radius at $\pi(\bar{p})$ is bigger than or equal to $e^{-\nu(M)}\text{\inj}_{M_{0}}$, where we set
\[
\boxed{\text{\inj}_{M_{0}}:=\inf_{x\in M_{0}}\inj_{x},}
\]
\emph{cf}. Definition \ref{def1}.
Hence 
if $e^{-\nu(M)} \inj_{M_{0}}\geq 1$, we have 
\[
\max_{h\in S(\mathcal{H}^k(M))}|h(p)|^2 \leq \frac{K_n}{e^{-\nu(M)} \inj_{M_{0}}},
\]
for some constant $K_n$ depending on the dimension $n$ and $g_{\IR}$ only. This follows from Price inequalities for harmonic forms on real-hyperbolic manifolds (see \cite[Corollary 108]{DS17}). Thus, under this assumption on $e^{-\nu(M)}\text{\inj}_{M_{0}}$, we have
\begin{align}\label{henmay}
b_2^k(M)\leq D_n(\nu(M)+1)^2e^{-(n-2)\nu(M)}\frac{ Vol(M_{1}\setminus M_{0})}{\inj_{M_0} }. 
\end{align}
In particular, under this restrictive injectivity assumption, we find that 
\begin{align}
\frac{b_2^{\frac{n-1}{2}}(M)}{Vol(M)}\leq \frac{L(n)}{\inj_{M_0}}
\end{align}
for some positive constant $L(n)$ depending on the dimension only. This is analogous to Equation \eqref{critdeg} in Theorem \ref{real hyperbolic} for compact real-hyperbolic manifolds, where the global injectivity radius of the compact hyperbolic manifold is replaced with the injectivity radius, $\text{\inj}_{M_{0}}$, of the thick part of the manifold with cusps.\\

We now study the remaining case $e^{-\nu(M)} \inj_{M_{0}}< 1$.
Consider the torus $T^a =  \Lambda_a\backslash\IR^{n-1}$. We recall some invariants of lattices of translations acting on Euclidean spaces. For $i\in \{1, ..., n-1\}$, set
\begin{align}\label{defrel}
\lambda_i(\Lambda_a):= \inf\{r:\text{dim span }\Lambda_a\cap\bar B(0,r)\geq i\},
\end{align}
Then $\lambda_1(\Lambda_a)$ is the injectivity radius of $T^a$, and $\delta_{\Lambda_a} = \lambda_1(\Lambda_a^*)$. We recall the lattice relation for a rank $n-1$ lattice:
\begin{align}\label{lattice}
\delta_{\Lambda_a} \lambda_{n-1}(\Lambda_a)\geq 1. 
\end{align}
 To see this, let $v\in \Lambda^*_a$ such that $|v|=\delta_{\Lambda_a}=\lambda_1(\Lambda^*_a)$. Consider any set $w_1, ..., w_{n-1}$ of $n-1$ linearly independent vectors in $(\Lambda^*_a)^*=\Lambda_{a}$. By construction of the dual lattice, there exists $i\in \{1, ..., n-1\}$ such that $\langle w_i, v \rangle\neq 0$, so that 
\[
\langle w_i, v \rangle\in \IZ \quad \Rightarrow \quad |w_{i}|\geq \frac{1}{|v|},
\]
and the proof of \eqref{lattice} is complete.  Next, for $y\in \Omega^a_{(\nu(M)-1)\nu(M)}$ with $a\not\in I(\nu(M))$, we have by definition
\[
\frac{1}{\delta_{\Lambda_a}}\geq \frac{2\pi e^{(\nu(M)-2)}}{4(\nu(M)+1)^2}.
\]
Then by \eqref{lattice}, we have 
\begin{align}\label{lattice2}  
\lambda_{n-1}(\Lambda_a)\geq \frac{e^{\nu(M)}}{5(\nu(M)+1)^2}. 
\end{align}
Let $\tilde C_a$ denote the universal cover of $C_a$, and let $q:\tilde C_a\to C_a$ denote the projection map. Recall that our near term goal is to estimate from above $|h(\bar p)|^2$ for 
\[
\pi(\bar{p})\in C_a\cap (M_{\nu(M)}\setminus M_{\nu(M)-1}), \quad  a\not\in I(\nu(M)),
\]
so that we can exploit \eqref{fund}. We will use elliptic estimates to obtain this pointwise bound from an integral bound for $|h|^2$ in a geodesic ball about $\pi(\bar p)$. Such estimates become worse as the radius of the ball decreases. Hence we will lift $h$ to $\tilde C_a$, where we can take the radius to be large. We next bound the $L^2$-norm of the lift in terms of the $L^2$-norm of $h$. This entire discussion is  nontrivial only when the ambient lattices  are far from square, as the estimate of the number of lattice points in a ball becomes fairly trivial in that case.  In order to estimate these integrals on $\tilde C_a$, we need to bound the number of elements in $q^{-1}(y)\cap B(o,\frac{R}{2}) $, $R\leq 1$ to be determined, for any $o\in \tilde C_a\cap q^{-1}(M_{\nu(M)}\setminus M_{\nu(M)-1})$. Since this intersection is empty unless there exists $y_0\in q^{-1}(y)$ such that $B(o,\frac{R}{2}) \subset B(y_0,R)$, it suffices to estimate  $q^{-1}(y)\cap B(y_0,R)$. Write for $y\in T_s^a$,  $s\in [\nu(M)-1-\frac{R}{2},\nu(M)+\frac{R}{2}]$,
\[
q^{-1}(y)\cap B(y_0,R) = \{y_0 + v: v\in \Lambda_a\cap B(0,e^{s}R)\}.
\]
Here we recall  that $\Lambda_a$ is a lattice in $T_0$, and the metric rescales the lattice by a factor of $e^{-s}$ in $T_s^a$. For any $R$ such that  $e^sR< \lambda_{n-1}(\Lambda_a)$, we have $\{v\in \Lambda_a\cap B(0,e^{s}R)\}$ is a subset of an $(n-2)$-dimensional linear subspace, say $\Sigma_{n-2}^a$. So, we now restrict $R$ to satisfy 
$$e^{\nu(M)+1}R<\lambda_{n-1}(\Lambda_a).$$
Hence we can estimate the cardinality of the set $\{v\in \Lambda_a\cap B(0,e^{s}R)\}$ as follows. Define the larger set $S:= \Lambda_a\cap B(0, \lambda_{n-1}(\Lambda_a))$. Given two distinct vectors  $v$ and $w$  in $S$, we have 
\[
B\Big(v,\frac{\lambda_1(\Lambda_a)}{2}\Big)\cap B\Big(w,\frac{\lambda_1(\Lambda_a)}{2}\Big)=\emptyset
\]
so that
\begin{align}\nonumber
&\sum_{v\in S}vol\big(B\Big(v,\frac{\lambda_1(\Lambda_a)}{2}\Big)\cap \Sigma_{n-2}^a)\leq vol(B\Big(0,e^sR +\frac{\lambda_1(\Lambda_{a})}{2} \Big)\cap \Sigma_{n-2}^a)\nonumber\\
& \Longrightarrow |S| \lambda_1(\Lambda_a)^{n-2}\leq C_n\Big(e^{\nu(M)}R+\frac{\lambda_1(\Lambda_a)}{2}\Big)^{n-2} \nonumber, 
\end{align}
for some positive constant $C_n$. We therefore have the estimate
\begin{align}\label{Nbound}
|q^{-1}(y)\cap B(y_0,R)| \leq C_n\lambda_1(\Lambda_a)^{2-n}\Big(e^{\nu(M)}R+\frac{\lambda_1(\Lambda_a)}{2}\Big)^{n-2}.
\end{align}
Let $h\in \mathcal{H}^{\frac{n-1}{2}}(M)$. Standard elliptic estimates (see for example \cite[Lemma 51]{DS17}) give the following estimate for some $a(n)>0$, for $x\in C_a$, and for $x_0\in q^{-1}(x)$: 
\begin{align}\nonumber
|h|^2(x) = |q^*h|^2(x_0)\leq a(n)\Big(1+\frac{2}{R}\Big)^n\|q^*h\|^2_{L^2\big(B_{\frac{R}{2}}(x_0)\big)}.
\end{align}
By \eqref{Nbound}, for $x\in  C_a \cap (M_{\nu(M)}\setminus M_{\nu(M)-1})$, the map $q: B_{\frac{R}{2}}(x_0)\to C_a$ is at most
\[
\Big[C_n\lambda_1(\Lambda_a)^{2-n}\Big(e^{\nu(M)}R+\frac{\lambda_1(\Lambda_a)}{2}\Big)^{n-2}\Big]\quad \text{to} \quad 1,
\]
where by $[\cdot]$ we denote the integer part. Therefore 
\begin{align}\nonumber
\|q^*h\|^2_{L^2\big(B_{\frac{R}{2}}(x_0)\big)}\leq C_n\lambda_1(\Lambda_a)^{2-n}\Big(e^{\nu(M)}R+\frac{\lambda_1(\Lambda_a)}{2}\Big)^{n-2}\|h\|^2_{L^2(M)},
\end{align}
so that there exists some $D(n)>0$ such that 
\begin{align}\label{hx}
|h(x)|^2  \leq D(n)\Big(1+\frac{2}{R}\Big)^n \Big(e^{\nu(M)}\lambda_1(\Lambda_a)^{-1}R+\frac{1}{2}\Big)^{n-2}\|h\|^2_{L^2(M)}.
\end{align}
Now combining Equations \eqref{fund} and \eqref{hx}, we obtain for some constant $D_2(n)>0$, 
\begin{align}\label{refund}
& b^k_2(M) \\ \nonumber
&\leq D_2(n)e^{-\nu(M)}(\nu(M)+1)^2 Vol(M_{1}\setminus M_{0})\Big(1+\frac{2}{R}\Big)^n\Big(\lambda_1(\Lambda_a)^{-1}R+\frac{e^{-\nu(M)}}{2}\Big)^{n-2} .
\end{align}
We now choose   
$$R:= \frac{1}{20(\nu(M)+1)^2} .$$
By Equation \eqref{lattice2} this $R$ satisfies our earlier constraint: 
$e^{\nu(M)+1}R<\lambda_{n-1}(\Lambda_a).$
Hence Equation \eqref{refund} yields
\begin{align}\label{2xrefund}
& b^k_2(M) \\ \nonumber 
&\leq D_2(n)e^{-\nu(M)}Vol(M_{1}\setminus M_{0})\Big(1+40 (\nu(M)+1)^2\Big)^{n+1}\Big(\frac{\lambda_1(\Lambda_a)^{-1}}{ 20 (\nu(M)+1)^2}+\frac{e^{-\nu(M)}}{2}\Big)^{n-2}.
\end{align}

We summarize this discussion with a theorem.

\begin{theorem}\label{criticalreaL}
Let $ M$ be a  $2k+1$ dimensional, complete finite volume real-hyperbolic manifold. Write $M$ as a disjoint union 
\[
M= M_0\cup(\cup_aC_a),
\]
where each $C_a$ is a cusp.  Then there exists $B(k)>0$ depending only on $k$ so that if  
$\inj_{M_0}\geq 1$ then 
\begin{align}\label{critcuspdeg}
b_2^k(M)\leq B(k) \frac{Vol(M)}{\inj_{M_0}}.
\end{align} 
\end{theorem}

\begin{proof}
If $\nu(M) = 0$, we have  
\[
\binom{n}{k}\int_{\pi^{-1}(M_{0})}\int_{S(\mathcal{H}^k(M))}|h(\bar{p})|^2d\sigma_hd\bar{v}_{\bar{p}}\geq\frac{1}{2\zeta(2)}>\frac{1}{4},
\]
so that there exists $\bar{p}\in\pi^{-1}(M_0)$ such that $\binom{n}{k}\int_{S(\mathcal{H}^k(M))}|h(p)|^2d\sigma_h\geq\frac{1}{4Vol(M_0)}$. We now apply Lemma \ref{K-first} to obtain for $h\in S(\mathcal{H}^k(M))$: 
\begin{align}\label{p1}
\int_{S_r(\pi(\bar{p}))}\Big(\frac{1}{2}-\mu_h\Big)|h|^2d\sigma = \int_{B_r(\pi(\bar{p}))}\mu_h|h|^2dv,
\end{align}
for any $r\leq inj_{\pi(\bar{p})}$. Since $\mu_h(0)<\frac{1}{2}$ by \cite[Lemma 18]{DS17},  Equation \eqref{p1} implies $\mu_h(r)\in (0,\frac{1}{2})$ for all $r$. Consequently, $\int_{S_r(\pi(\bar{p}))}\Big(\frac{1}{2}-\mu_h\Big)|h|^2d\sigma$ is monotonically increasing. Then we have for $\frac{1}{2}\leq \tau<\inj_{M_0}$  
\begin{align}\label{ok}
&\int_{B_\tau(\pi(\bar{p}))}|h|^2dv\geq  \int_{\frac{1}{2}}^\tau\int_{S_r(\pi(\bar{p}))}|h|^2d\sigma dr>2 \int_{\frac{1}{2}}^\tau\int_{S_r(\pi(\bar{p}))}\Big(\frac{1}{2}-\mu_h\Big)|h|^2d\sigma dr\nonumber\\
&>2 \int_{\frac{1}{2}}^\tau\int_{S_{\frac{1}{2}}(\pi(\bar{p}))}\Big(\frac{1}{2}-\mu_h\Big)|h|^2d\sigma dr = 2\Big(\tau-\frac{1}{2}\Big)\int_{B_{\frac{1}{2}}(\pi(\bar{p}))}\mu_h|h|^2dv.
\end{align}
In the proof of \cite[Corollary 108]{DS17}, we show that for $r\in [0,1]$, $\mu_h(r)\geq \beta(n)$, for some $\beta(n)>0$ independent of $h$ and $M$. Hence Equation \eqref{ok} coupled to the elliptic estimate given in \cite[Lemma 51]{DS17} yield for $\nu(M)=0$, and for some positive constants $c(n)$ and $d(n)$,
\begin{align}\nonumber
&b_2^k(M)\leq c(n)Vol(M_0)|h(\pi(\bar{p}))|^2\leq d(n)Vol(M_0)\int_{B_{\frac{1}{2}}(\pi(\bar{p}))} |h|^2dv \nonumber\\
&\leq \frac{2 d(n)\beta(n)^{-1}}{\big(\inj_{M_0}-\frac{1}{2}\big)}\int_{B_{\inj_{M_0}}(\pi(\bar{p}))}|h|^2dv \cdot Vol(M_0)\leq \frac{2 d(n)\beta(n)^{-1}}{\big(\inj_{M_0}-\frac{1}{2}\big)}\cdot Vol(M_0)\nonumber .
\end{align}
Finally if $\nu(M)>0$, then Equation \eqref{2xrefund} bounds from above $b_2^k$ with a constant multiple of $\frac{Vol(M_0)}{\inj_{M_0}}$,
and the result follows. 
\end{proof}

\section{$L^2$-Cohomology on a Tower of Coverings}\label{towers} 
 
We now study the growth of $L^2$-cohomology on towers of real- and complex-hyperbolic manifolds with cusps. 
 
\subsection{Fattening the Thick Part}\label{Fat}

We show that, up to finite cover, we can fatten the injectivity radius of the thick part of the quotient of a Hadamard manifold with residually finite fundamental group. This is the analog for a manifold with cusps of Theorem 2.1 in \cite{DW78}. 

Let $(M, g)$ be a  complete finite volume Riemannian manifold such that $-b^{2}\leq\sec_{g}\leq-a^2$, with $a, b\neq 0$.  Denote by $(\widetilde{M}, \tilde{g})$ the Riemannian universal cover of $(M, g)$.  Let $\Gamma=\pi_{1}(M)$ be the associated lattice in $\text{Iso}(\widetilde{M})$, so that  $M=\Gamma\backslash\widetilde{M}$. Assume that $\Gamma$ is residually finite. Let $\{\Gamma_k\}_{k\in\IN}$ be a cofinal filtration  of $\Gamma=:\Gamma_0$ by finite index normal subgroups. We define the continuous $\Gamma$-invariant function
$
d_{\Gamma}: \widetilde{M}\rightarrow [0, \infty)$ by
\[
d_{\Gamma}(p)=\min\{d_{\tilde{g}}(p, \gamma p):  \gamma\in\Gamma, \gamma\neq 1\}.
\]
Thus $d_{\Gamma}(p)$ is twice the injectivity radius of $\omega(p)$ in $M$, where $\omega:\widetilde{M}\rightarrow M$ is the universal covering map. Given $\delta>0$, define the closed subset $\widetilde{M}_{\delta}\subset \widetilde{M}$ by
\[
\widetilde{M}_{\delta}:=\{q\in\widetilde{M}\quad |\quad d_{\Gamma}(q)\geq \delta\}.
\]
Also, define the associated compact subset of $M$: 
\[
M_{\delta}:=\{p\in M\quad | \quad 2\cdot \inj_p\geq\delta\}.
\] 
Then $\omega(\widetilde{M}_{\delta})= M_{\delta}$. Let $\{q_k:M_{k}\to M\}_{k\in\IN}$ be the sequence of regular Riemannian coverings of $M$ associated to the cofinal filtration $\{\Gamma_{k}\}_{k\in\IN}$.   For each $k$, define the numerical invariant
\begin{align}\notag
r_{k, \delta}:=\min\{ d(z, \gamma_{k}z)\quad | \quad z\in \widetilde{M}_{\delta},\quad\gamma_{k}\in \Gamma_{k},\quad\gamma_{k}\neq 1 \}.
\end{align}
Notice that by definition we have 
\[
r_{0, \delta}=\delta, \text{ and } \quad r_{k, \delta^{\prime}}\leq r_{k, \delta}
\]
for any $\delta^\prime\leq\delta$ and for any $k$.

\begin{lemma}\label{GettingFat}
For any $k\geq 0$, denote by $\omega_k : (\widetilde{M}, \tilde{g}) \rightarrow (M_k, q^{*}_{k}(g))$ the Riemannian universal  covering map. For every $\delta>0$ and for any $z\in\widetilde{M}_{\delta}$, we have 
\begin{align}\label{iso}
\omega_{k} : B\Big(z; \frac{r_{k, \delta}}{2}\Big)\cap \widetilde{M}_{\delta}\rightarrow \omega_{k}\Big(B\Big(z; \frac{r_{k, \delta}}{2}\Big)\cap\widetilde{M}_{\delta}\Big)
\end{align}
is an isometry. Finally, we have  
\begin{align}\label{infinite}
\lim_{k \to \infty}r_{k, \delta}=\infty.
\end{align}
\end{lemma}

\begin{proof}
If Equation \eqref{infinite} does not hold, there exist infinite sequences $z_{k}\in \widetilde{M}_{\delta}$ and $\gamma_{k}\in \Gamma_{k}\setminus\{1\}$ such that $d(z_{k}, \gamma_{k} z_{k})\leq 2N$ for some positive constant $N$. Let $D\subset \widetilde{M}_{\delta}$ be a connected open set such that $\omega: D\rightarrow M_{\delta}$ is injective and $\omega: \bar{D}\rightarrow M_{\delta}$ is surjective, where $\bar{D}$ is the closure of $D$ in $\widetilde{M}_{\delta}$. Thus for all $k$, there exists $g_{k}\in \Gamma$ such that $g_{k} z_{k}\in\bar{D}$. Define $z^\prime_{k}=g_{k} z_{k}$ and $\gamma^\prime_{k}=g_{k}\gamma_{k}g^{-1}_{k}$. Since $\Gamma_{k}$ is a normal subgroup of $\Gamma$, we have that $\gamma^\prime_k\in\Gamma_k$. By compactness of $\bar{D}$, there exists a subsequence $\{z'_{k_{j}}\}$ converging to a point $\bar{z}\in\bar{D}$. Now since 
\[
d(z^\prime_{k}, \gamma^\prime_{k}z^\prime_{k})=d(g_{k}z_{k}, g_{k}\gamma_{k}z_{k})=d(z_{k}, \gamma_{k}z_{k}),
\]
we have that 
\begin{align}\notag
d(\bar{z}, \gamma^\prime_{k_{j}}\bar{z})\leq 2d(\bar{z}, z^\prime_{k_{j}})+2N. 
\end{align}
Since $d(\bar{z}, z^\prime_{k_{j}})\rightarrow 0$, we conclude that, up to a subsequence, 
$\gamma^\prime_{k_{j}}\bar{z}$ converges to a point $w\in \overline{B(\bar{z}; 2N)}\cap\widetilde{M}_{\delta}$ for some $\epsilon>0$. This implies that 
\begin{align}\notag
\omega(\bar{z})=\omega(\gamma^\prime_{k_{j}}\bar{z})\longrightarrow \omega(w).
\end{align}
Thus, there exists $\gamma\in\Gamma$ such that $\gamma w=\bar{z}$. We therefore conclude 
\begin{align}\notag
(\gamma^\prime_{k_{j}}\cdot \gamma) w=\gamma^\prime_{k_{j}}\bar{z}\longrightarrow w.
\end{align}
Now the action of $\Gamma$ on $\widetilde{M}$ is properly discontinuous, so that $\gamma^\prime_{k_{j}}\cdot \gamma=\{1\}$ for all $j$ sufficiently large. Thus, we must have $\gamma=\{1\}$ which then implies the contradiction $\gamma^\prime_{k_{j}}=\{1\}$. The proof of \eqref{infinite} is then complete. Equation \eqref{iso} simply follows from the definition of $r_{k, \delta}$.
\end{proof}

We can now present the analog for manifold with cusps of Theorem 2.1 in \cite{DW78}.

\begin{theorem}\label{fat}
	Let $(M:=\Gamma\backslash\widetilde{M}, g)$ be a complete finite volume  quotient of a Hadamard manifold $(\widetilde{M}, \tilde{g})$. Assume $\Gamma=\pi_1(M)$ is residually finite. Given a cofinal filtration $\{\Gamma_{i}\}$ of $\Gamma$,  denote by $\omega_{i}: M_{i}\rightarrow M$ the regular Riemannian cover of $M$ associated to $\Gamma_{i}$ equipped with the pull-back metric, say $g_{i}$. Decompose $M$ as a disjoint union
	\[
	M=M_{0}\cup(\cup_j C_j),
	\]
	where $M_0$ is compact manifold with boundary and each $C_j$ is a cusp. Define $M^i_{0}:=\omega^{-1}_{i}(M_{0})$, and let $\inj_{M^i_{0}}:=\inf_{x\in M^{i}_{0}}\inj_{g_{i}}(x)$. We then have
	\begin{align}\label{goinginfinity}
	\lim_{i\to\infty}\inj_{M^i_{0}}=\infty.
	\end{align}
\end{theorem}
\begin{proof}
Let $\{\delta_{n}\}$ be a sequence of real numbers converging to zero, and let $\{M_{n}\}$ be a sequence of closed subset in $M$ defined by
\[
M_{n}:=\{p\in M\quad | \quad 2\cdot \inj_p\geq\delta_n\}.
\] 
We then have that
\[
\lim_{n\to\infty}d_n=\infty,
\]
where $d_{n}:=dist(\partial M_{0}, \partial M_{n})$. Given an integer $n\geq 1$, by Lemma \ref{GettingFat} there exists an integer $k_n$ such that
\[
r_{k, \delta_n}\geq d_{n},
\]
for any $k\geq k_{n}$. We therefore have that
\[
\inj_{M^k_{0}}\geq d_{n},
\]
for any $k\geq k_n$. By letting $n\to\infty$ and recalling that by construction $d_n\to\infty$, we conclude that Equation \eqref{goinginfinity} is satisfied.
\end{proof}

\subsection{Asymptotic Behavior of $L^2$-Cohomology}

In this section, we study the $L^2$-cohomology of complete finite volume hyperbolic manifolds on towers of coverings. We first  extend Theorem \ref{real hyperbolic} to real-hyperbolic manifolds with cusps. 
 
\begin{theorem}\label{realhypcusps}
	Let $(M^{n}:=\Gamma\backslash\textbf{H}^{n}_{\IR}, g_{\IR})$, with $\Gamma$ co-finite volume. Given a cofinal filtration $\{\Gamma_{i}\}$ of $\Gamma$,  denote by $\pi_{i}: M_{i}\rightarrow M$ the regular Riemannian cover of $M$ associated to $\Gamma_{i}$. Decompose $M^{n}$ as a disjoint union
	\[
	M=M_{0}\cup(\cup_j C_j)
	\]
	where $M_0$ is a compact manifold with boundary and each $C_j$ is a cusp. Define $M^i_{0}:=\pi^{-1}_{i}(M_{0})$, and assume $\inj_{M^i_{0}}:=\inf_{x\in M^i_{0}} inj_{g_{i}}(x)\geq 1$. For any integer $1\leq k<\frac{n-1}{2}$,  there exists a positive constant $c_{1}(n, k)$ such that
	\begin{align}
	\frac{b^{k}_2(M_{i})}{Vol(M_{i})}\leq c_{1}(n, k)V_{min}(M^i_{0})^{-\frac{(n-1-2k)}{n-1}}.
	\end{align}
In particular, the sub volume growth of the Betti numbers along the tower of coverings is exponential in $\inj_{M^i_{0}}$.
For $n=2k+1$ there is a positive constant $c_{2}(n, k)$ such that
	\begin{align}\label{critdeg}
	\frac{b^{k}_2(M_{i})}{Vol(M_{i})}\leq \frac{c_{2}(n)}{\inj_{M^i_{0}}}.
	\end{align}
\end{theorem}
\begin{proof}
	Combine Corollary \ref{rhypcusps} and Theorem \ref{criticalreaL} with Theorem \ref{fat}.  
\end{proof}

We have the following analogous result for complex-hyperbolic manifolds with cusps. 

\begin{theorem}\label{complhypcusps}
Let $(M:=\Gamma\backslash\textbf{H}^{n}_{\IC}, g_{\IC})$ be complete finite volume. Given a cofinal filtration $\{\Gamma_{i}\}$ of $\Gamma$,  denote by $\pi_{i}: M_{i}\rightarrow M$ the regular Riemannian cover of $M$ associated to $\Gamma_{i}$. Decompose $M$ as a disjoint union
\[
M=M_{0}\cup(\cup_j C_j)
\]
where $M_0$ is a compact manifold with boundary and each $C_j$ is a cusp. Define $M^i_{0}:=\pi^{-1}_{i}(M_{0})$, and assume $\inj_{M^i_{0}}:=\inf_{x\in M^i_{0}} inj_{g_{i}}(x)\geq 1$. For any integer $1\leq k<n$,  there exists a positive constant $c_{1}(n, k)$ such that
\begin{align}
\frac{b^{k}_2(M_{i})}{Vol(M_{i})}\leq c(n, k)V_{min}(M^i_{0})^{\frac{k-n}{n}}.
\end{align}
In particular, the decay of $\frac{b^{k}_2(M_{i})}{Vol(M_{i})}$  along the tower of coverings is exponential in $\inj_{M^i_{0}}$.
\end{theorem}
\begin{proof}
	Combine Corollary \ref{L2C} with Theorem \ref{fat}.  
\end{proof}

\subsection{Noncompact congruence subgroup quotients}\label{sharp}

In this subsection, we adapt and sharpen the arguments of \cite[Sections 2 and 4]{Sarnak} to the noncompact case.   We realize $SO(n,1)$ and $SU(n,1)$  as the stabilizers, $SO(Q)$ and $SU(H)$ respectively, of integral quadratic (respectively hermitian)
 forms of signature $(n,1)$:
$Q(x) := \sum_{1\leq i,j\leq n}b_{ij}x_ix_j$, and $H(z) := \sum_{1\leq i,j\leq n}h_{ij}z_i\bar z_j$. Let $b$ and $h$ denote the corresponding matrices.   Let $I_{n+1}$ denote the identity matrix in $GL_{n+1}$.  
Let $\Gamma_\IR(n,q)$ and $\Gamma_\IC(n,q)$ denote the arithmetic subgroups of $SO(n,1)$, respectively  $SU(n,1)$ defined
by 
$$\Gamma_\IR(n,q):= SO(n,1)\cap \{g\in GL(n+1,\IZ): g \equiv I_{n+1}  \text{ (mod }q)\},$$
and 
$$\Gamma_\IC(n,q):= SU(n,1)\cap \{g\in GL(n+1,\IZ+i\IZ): g \equiv I_{n+1}  \text{ (mod }q)\}.$$
Let $M_\IR(n,q) := \Gamma_\IR(n,q)\backslash \textbf{H}^{n}_{\IR}$, and $M_\IC(n,q) := \Gamma_\IC(n,q)\backslash \textbf{H}^{n}_{\IC}$. 

On a locally symmetric space, $\Gamma\backslash X$, the injectivity radius at a point $\Gamma p$ is 
$\frac{1}{2}\inf_{\gamma\in\Gamma}d(p,\gamma p).$ Embedding the symmetric space $X$ in $GL(N)/SO(N)$, for some $N$ ($= n+1$ in our case), Sarnak and Xue calculate in \cite[Section 2]{Sarnak} that $d(SO(N),\gamma SO(N))$ is greater than or equal to 
$(\sum_j\ln(\lambda_{j=1}^N)^2)^{\frac{1}{2}}$, where $\{\lambda_j\}_{j=1}^N$ denotes the eigenvalues of $\gamma^t\gamma$. 

Let $|w|_{\infty} $  denote the supremum norm ($=sup_{v\neq 0}\frac{||w(v)||}{||v||}$) for a matrix $w$.
The following algebraic lemma will be used to estimate injectivity radii. 
\begin{lemma}\label{thankmike}If
 $\gamma\in \Gamma_\IR(n,q)\setminus \{I_{n+1}\}$ and $q^2\geq 2|b|_{\infty} $,  then the largest eigenvalue  of $\gamma^t\gamma$ is greater than or equal to $(\frac{q^2-2|b|_{\infty}}{2\text{det}(b)})^2,$
respectively $(\frac{q^2-2|h|_{\infty}}{2\text{det}(h)})^2.$ 
If  $\mu\in \Gamma_\IC(n,q)\setminus \{I_{n+1}\}$,  and  $q^2\geq 2|h|_{\infty} $, then the largest eigenvalue  of $\mu^t\mu$ is greater than or equal to $(\frac{q^2-2|h|_{\infty}}{2\text{det}(h)})^2.$  
\end{lemma} 
\begin{proof}
	We treat the $\Gamma_\IR(n,q)$ case. The $\Gamma_\IC(n,q)$ case is identical. Let $\gamma\in \Gamma_\IR(n,q).$
Write $\gamma = I_{n+1}+ qh_1$, $h_1$ integral. Then 
\begin{align}\label{q1}b = \gamma^tb\gamma = (I_{n+1}+qh_1^t)b(I_{n+1}+qh_1)=  b+ q(bh_1 +h_1^tb) + q^2h_1^t b h_1.
\end{align}
Then we can write $bh_1= A_1+qbh_2$, with $A_1$ skew symmetric and (half-) integral and $bh_2$ symmetric and (half-) integral. Inserting this into \eqref{q1} yields 
\begin{align}\label{q2} 2 bh_2 =A_1bA_1-q(h_2^tA_1-A_1h_2) -q^2h_2^t bh_2.
\end{align}
Since $A_1$ is skew and $b$ is symmetric and invertible, $A_1bA_1\not = 0 $ unless $A_1 = 0$.  
In particular, if the integral matrix $A_1bA_1$ vanishes, then   $h_2\not = 0$, by the nontriviality of $\gamma$. On the other hand, we see from \eqref{q2} that $h_2$ is nonzero if $A_1$ is nonzero. In either case,  $h_{2}$ cannot be zero, and we have
\[
b \gamma+(b\gamma)^t=2b + 2q^2  b h_2 .
\]
Hence $|b |_{\infty}|\gamma|_{\infty}\geq |b\gamma|_{\infty}\geq  \frac{q^2}{2}-|b|_{\infty},$ and the claim follows easily.
\end{proof}

 Fix  decompositions 
$$M_\IR(n,1) = M_\IR(n,1 )_0\cup_j C_j^1, \text{ and } M_\IC(n,1) = M_\IC(n,1 )_0\cup_j \tilde C_j^1,$$
where the $C_j^1$ and $\tilde C_j^1$ are cusps, and $M_\IR(n,1 )_0$ and $M_\IC(n,1 )_0$ are compact orbifolds with boundary. 
Fix similar decompositions  
$$M_\IR(n,q) = M_\IR(n,q )_0\cup_j C_j^q, \text{ and } M_\IC(n,1) = M_\IC(n,1 )_0\cup_j \tilde C_j^q,$$
where $M_\IR(n,q )_0$ and $M_\IC(n,q )_0$ project to $M_\IR(n,1 )_0$ and $M_\IC(n,1 )_0$ and the cusps project to the cusps under the natural map. 
Let $\delta$ and $\tilde \delta$ denote the diameters of fundamental domains for  $M_\IR(n,1 )_0$ and $M_\IC(n,1 )_0$ in their universal covers, respectively. 
\begin{lemma} \cite[Section 2]{Sarnak}\label{Sarnak}
There exists constants $G_b$ and $G_h$, depending on $b$ and $h$ so that for $q$ large, 
\begin{align}\label{injestq}
\inj_{M_\IR(n,q )_0} \geq 2\ln(q) -G_b-\delta,\text{ and }\inj_{M_\IC(n,q )_0} \geq 2\ln(q)-G_h -\tilde \delta.
\end{align}  
\end{lemma}
\begin{proof}
Let $\Gamma_{\IR}(n,q)p\in M_\IR(n,q )_0.$ Then 
\begin{align}\inj_{\Gamma_{\IR}(n,q)p} &\geq \inj_{\Gamma_{\IR}(n,q)SO(n)}-\delta\geq 
\frac{1}{2}\ln((\frac{q^2-2|b|_{\infty}}{2\text{det}(b)})^2)-\delta\nonumber\\
&= 2\ln(q) + \ln((\frac{1-2q^{-2}|b|_{\infty}}{2\text{det}(b)}))-\delta.
\end{align}
For  $q^2\geq 4 |b|_{\infty}+1$, set $G_b:= |\ln(\frac{1 }{4\text{det}(b)})|$, and the first claim follows. The proof of the hermitian case is identical. 
\end{proof}
We can now prove the following.
\begin{theorem}
For $k<\frac{n-1}{2},$ there exist  constants $A_G$ depending on  $Q$ and   $c_1(n,k) >0$ such that for $\Gamma_\IR(n,q)$ torsion free,
\begin{align}b_2^k(M_\IR(n,q))\leq A_Qc_1(n,k) Vol(M_\IR(n,q))^{1-\frac{4(n-1-2k)}{n(n+1)}}.
\end{align}
There exists constants $B_G$ depending on  $H$ and   $c(n,k) >0$ such that for $\Gamma_\IC(n,q)$ torsion free,
\begin{align}b_2^k(M_\IC(n,q))\leq B_Gc(n,k) Vol(M_\IC(n,q))^{1-\frac{4(n-k)}{(n+1)^2-1}}.
\end{align}
\end{theorem}
\begin{proof} The theorem is a  consequence of Theorems \ref{realhypcusps} and \ref{complhypcusps}, and Lemma \ref{Sarnak}. First use the injectivity radius estimates \eqref{injestq} to estimate for positive constants $\alpha_j$ independent of $q$, 
$V_{min}(M_\IR(n,q )_0)\geq \alpha_1 q^{2n-2}$ and $V_{min}(M_\IC(n,q )_0)\geq \alpha_2 q^{4n}$. Then note that for positive constants $\beta_j$ independent of $q$,
$Vol(M_\IR(n,q))\leq  \beta_1 q^{\frac{n(n+1)}{2}}$ and $Vol(M_\IC(n,q))\leq  \beta_2 q^{ (n+1)^2-1}$. These latter estimates follow from the embedding of $\Gamma_\IK(n,q)\backslash \Gamma_{\IK}(n,1)$, $\IK = \IR, \IC$ into the corresponding orthogonal (respectively unitary) groups with coefficients in a finite field.  Orders for these groups are computed, for example, in \cite[Section 1.6]{O}. 
\end{proof}

We remark that the restriction to torsion free $\Gamma_\IK(n,q)$ is made simply for convenience and is not essential.

\section{Normalized Number of Cusps and Topological Interpretation}\label{tinterpretation}

In Section \ref{towers}, we analyzed the growth of $L^2$-cohomology on towers of coverings of complete finite volume  real- and complex-hyperbolic  manifolds. We now address the problem of understanding the growth of de Rham cohomology on towers of coverings of such manifolds. This study requires an estimate on the number of cusps of hyperbolic manifolds along towers of coverings. This is a geometric problem of independent interest.

\subsection{Normalized Number of Cusps}\label{NNC}

We start with some general remarks concerning complete finite volume quotients of Hadamard manifolds. Given a finite volume non-compact manifold $(M, g)$  with pinched sectional curvature,
\begin{align}\label{hadamard}
-b^2\leq\sec_g\leq -a^2<0. 
\end{align}
we define the volume normalized number of cusps (or simply the \emph{normalized number of cusps}) to be:
\begin{align}\label{cuspsratio}
\boxed{R_g(M):=\frac{N_{c}(M)}{Vol(M)},}
\end{align}
where $N_{c}(M)$  denotes the the number of cusps of $M$.
By foundational work of Eberlein \cite{Ebe80}, $N_c$ (the number of topological ends) is necessarily finite. In particular $\pi_1(M)$ is finitely  presented. On the other hand, counterexamples of Gromov \cite{Gro78} show this is not the case if the pinching condition is weakened to $a=0$, see also Example 2 page 459 in \cite{Ebe80}. This is true already in dimension $n=3$, and in all of these examples the fundamental group is infinitely generated.\\

Now,   Eberlein's result simply tells us that $N_c(M)<\infty$, but the actual upper bound  may depend upon $(M, g)$ itself.  With some extra work, one can show that the upper bound  depends upon the Margulis constant associated to $(M, g)$. Such coarse bounds are usually far from optimal, and the Margulis constant may be hard to estimate.  In many concrete geometric topology questions, it is then useful to derive effective uniform upper bounds on $R_{g}$ for certain specific classes of complete finite volume quotients of Hadamard manifolds. This concrete problem is particularly well-studied in the case of finite volume quotients of $(\textbf{H}^{n}_{\IR}, g_{\IR})$ and $(\textbf{H}^{n}_{\IC}, g_{\IC})$. Interestingly, one may study this problem with a wildly different set of techniques ranging from topology to complex algebraic geometry. We refer the interested reader to the classical references of Parker \cite{Parker} and Kellerhals \cite{Kel98}, and to the more recent literature in \cite{Hwang}, \cite{DD14}, \cite{DD15}, \cite{DD17}, and \cite{BT18}. In all of these instances, one seeks the \emph{smallest} explicitly computable number $c(n, \IK)$ such that
\[
R_{g_{\IK}}(\Gamma\backslash\textbf{H}_{\IK}^{n})\leq c(n, \IK),
\]
for \emph{all} non-uniform torsion free lattices $\Gamma\leq\text{Iso}(\textbf{H}_{\IK}^{n})$, where $\IK$ is either $\IR$ or $\IC$.

When $\IK=\IR$ and $n=3$, there is an intriguing connection between normalized Betti numbers and  $R_{g_{\IR}}$. Namely, it is a consequence of basic $3$-manifold  theory (see for example \cite[Lemma 3.5]{Hatcher3}) that for any complete finite volume $\Gamma\backslash\textbf{H}_{\IR}^{3}$
\begin{align}\label{3bound}
b_{1}(\Gamma\backslash\textbf{H}_{\IR}^{3})\geq N_{c}(\Gamma\backslash\textbf{H}_{\IR}^{3}) \quad \Longrightarrow \quad 0<R_{g_{\IR}}(\Gamma\backslash\textbf{H}_{\IR}^{3})\leq \frac{b_{1}(\Gamma\backslash\textbf{H}_{\IR}^{3})}{Vol(\Gamma\backslash\textbf{H}_{\IR}^{3}))}.
\end{align}
Thus, given a cofinal filtration $\{\Gamma_{i}\}$ of $\Gamma$,  combining L\"uck's approximation theorem \cite{Luck} with Dodziuk's vanishing \cite{Dodziuk} we obtain
\begin{align}\label{going0}
\lim_{i\to\infty}\frac{b_{1}(\Gamma_i\backslash\textbf{H}_{\IR}^{3})}{Vol(\Gamma_i\backslash\textbf{H}_{\IR}^{3}))}=0 \quad \Longrightarrow \quad \lim_{i\to\infty} R_{g_{\IR}}(\Gamma_i\backslash\textbf{H}_{\IR}^{3})=0.
\end{align}
We observe that in order to apply the main approximation theorem in \cite{Luck}, we simply need $\pi_{1}(\Gamma\backslash\textbf{H}_{\IR}^{3})$ to be finitely presented and residually finite, and this is certainly the case as $Vol(\Gamma\backslash\textbf{H}_{\IR}^{3})$ is assumed to be finite.

We next  generalize this result to show that the normalized number of cusps goes \emph{monotonically} to zero along a tower associated to a cofinal filtration of the fundamental group of any finite volume hyperbolic manifold. Now for general rank one locally symmetric spaces, one cannot expect a bound as in Equation \eqref{3bound} to hold. Indeed   \cite{DiS17} gives an explicit sequence of complex hyperbolic surfaces  with cusps with first Betti number equal to 2 and diverging number of cusps. Thus, unlike the case of hyperbolic $3$-manifolds, the convergence to zero of the normalized number of cusps cannot be derived  from L\"uck's approximation and vanishing of $L^2$-Betti numbers. Our approach also has the advantage of producing an \emph{effective} estimate on the rate of convergence of $R_{g_{\IK}}$. The following lemma is stated for complete finite volume quotients of rank one symmetric spaces of non-compact type, but it can be generalized to complete finite volume quotients of Hadamard manifolds with residually finite fundamental group. We do not develop that generalization here.

\begin{lemma}\label{normalizedcusps}
Let $(M^{n}:=\Gamma\backslash\textbf{H}_{\IK}^{n}, g_{\IK})$ be a complete finite volume  hyperbolic manifold with $\IK=\IR, \IC, \IH$, and $\IO$. Given a cofinal filtration $\{\Gamma_{i}\}$ of $\Gamma$, denote by $\omega_{i}: M_{i}\rightarrow M$ the regular Riemannian cover of $M$ associated to $\Gamma_{i}$ equipped with the pull-back metric say $g_{i}$. Decompose $M^n$ as a disjoint union
\[
M^n=M_{0}\cup(\cup_j C_j),
\]
where $M_{0}$ is a compact manifold with boundary and each $C_j$ is a cusp. Define $M^i_{0}:=\omega^{-1}_{i}(M_{0})$, and let $\inj_{M^i_{0}}:=\inf_{x\in M^i_{0}} inj_{g_{i}}(x)$. We then have
\begin{align}\label{333}
R_{g_{\IK}}(M_{i}):=\frac{N_{c}(M_{i})}{Vol(M_{i})}\leq Vol(B_{\frac{1}{2}\inj_{M^i_{0}}})^{-1}.
\end{align}
where $Vol(B_{\frac{1}{2}\inj_{M^i_{0}}})$ is the volume of a ball of radius $\frac{1}{2}\inj_{M^i_{0}}$ in $(\textbf{H}_{\IK}^{n}, g_{\IK})$.  
Moreover, we have the monotonicity
\[
R_{g_{\IK}}(M_{i})\geq R_{g_{\IK}}(M_{i+1}),
\]
for any $i$.
\end{lemma}
\begin{proof}
For any cusp $C^{i}_{a}$, there is a point $z_{a} \in C^{i}_{a}$ such that
\[
\inj_{z_a}=\frac{\inj_{M^i_{0}}}{2}.
\]
 Note that  $z_a$ is in the interior of  $C^i_a$. In particular,  for any point $q\in  \partial C^{i}_{a}$, we have
\[
d_{g_{i}}(q, z_a)\geq|\inj_{q}-\inj_{z_{a}}|>\frac{\inj_{M^i_{0}}}{2}. 
\]
We then have  
\[
 Vol(M_{i})> \sum_{a}Vol(B_{\frac{1}{2}\inj_{M^i_{0}}}(z_a))
=N_{c}(M_{i})Vol(B_{\frac{1}{2}\inj_{M^i_{0}}}),
\]
which implies \eqref{333}.
   
To show the monotonicity of $R_{g_{\IK}}$ along the tower, we argue as follows. Let 
\[
\tau: M^{\prime}\rightarrow M
\]
be \emph{any} covering of $M$ of degree say $\kappa\geq 2$. 
As each $C_j$ is covered by at most $\kappa$ disjoint cusps in $M^{\prime}$, we compute:
\[
R_{g_{\IK}}(M^\prime)=\frac{N_{c}(M^\prime)}{Vol(M^\prime)}\leq\frac{\kappa\cdot N_{c}(M)}{\kappa\cdot Vol(M)}=R_{g_{\IK}}(M),
\]
completing the proof. 
\end{proof}   
   
\begin{remark}
As shown in Lemma \ref{normalizedcusps}, $R_{g_{\IK}}$ is non-increasing in a tower of coverings. In particular, hyperbolic manifolds that do \emph{not} finitely cover any other hyperbolic manifold tend to have large $R_{g_{\IK}}$. Thus, hyperbolic manifolds with \emph{minimal} volume are good candidates on which to test the sharpness of any possible cusp count. Interestingly, this is exactly the case for the \emph{sharp} cusp count for complex-hyperbolic surfaces presented in \cite{DD14}. Indeed, such a bound is saturated by  an arithmetic  $4$-cusped complex-hyperbolic surface of minimal volume constructed by Hirzebruch in \cite{Hirzebruch}. 
\end{remark}

\subsection{Bounds for Cohomology on Towers}
   
We can now derive bounds for the de Rham cohomology of real- and complex-hyperbolic manifolds. Recall that for a complete finite volume  hyperbolic manifold, we have the following topological interpretation for the $L^2$-cohomology groups.

\begin{theorem}\label{Zucker} \cite[Theorems 6.2 and 6.9]{Zuc1} 
Let $ M^{n}$ be a complete finite volume real- or complex-hyperbolic manifold of real dimension $n$.  We have the following isomorphisms:
\begin{equation*} 
\mathcal{H}^k(M^n)= 
\left\{ \begin{array}{rl} & H^{k}(M^n; \IR), \quad \text{if } \quad k <\frac{n-1}{2};\\
&\\
& Im\big(H^{k}_{c}(M^n;\IR)\to H^{k}(M^n; \IR)), \quad \text{if } \quad k=\frac{n-1}{2},\frac{n}{2},\frac{n+1}{2};\\
&\\
&H^{k}_{c}(M^n), \quad \text{if } \quad k>\frac{n+1}{2};
\end{array} \right. 
\end{equation*}
where by $H^k_{c}(M^{n})$ we denote the cohomology with compact support in degree $k$ of $M^{n}$.
\end{theorem}

Combining Theorem \ref{Zucker} with our previous results, we obtain a satisfactory understanding of the Betti numbers of complex-hyperbolic manifolds and even dimensional real-hyperbolic manifolds in a tower of coverings obtained by a cofinal filtration.  For odd dimensional real-hyperbolic manifolds, we need further investigation in the critical degree. In particular, we need to keep into account the normalized number of cusps coefficient $R_{g_{\IR}}$, introduced in Section \ref{NNC}.  \\

Let $(M^{n}:=\Gamma\backslash\textbf{H}^{n}_{\mathbb R}, g_{\IR})$ be a complete finite volume real-hyperbolic manifold. Given a cofinal filtration $\{\Gamma_{i}\}$ of $\Gamma$, let $\pi_i: M_{i}\rightarrow M$ denote   the regular Riemannian cover of $M$ associated to $\Gamma_{i}$. Decompose $M^{n}$ as a disjoint union
\[
M=M_{0}\cup(\cup_j C_j),
\]
where $M_0$ is a compact manifold with boundary and each $C_j$ is a cusp. Define $M^i_{0}:=\pi^{-1}_{i}(M_{0})$. It follows that:
\[
M^i_{0}= M_{i}\setminus \bigcup^{N_{c}(M_{i})}_{s=1}C^i_{s}, \quad \partial M^i_{0}=\bigcup^{N_{c}(M_{i})}_{s=1}\p C^i_{s}.
\]
From the long exact sequence in cohomology, we have
\[
...\rightarrow H^{k}(M^i_{0}, \partial M^i_{0}; \IR)\rightarrow H^k(M_i; \IR)\rightarrow H^k(\cup_{s}C^i_{s}; \IR)\rightarrow H^{k+1}(M^i_{0}, \partial M^i_{0}; \IR)\rightarrow...,
\]
so that
\begin{align}\nonumber
\dim_{\IR}H^{k}(M_{i}; \IR)&\leq \dim_{\IR}Im\big(H^{k}(M^i_{0}, \partial M^i_{0}; \IR)\to H^{k}(M_{i}; \IR))   + L(n)N_{c}(M_i)\nonumber\\
&=\dim_{\IR}Im\big(H^{k}_c(M_i; \IR)\to H^{k}(M_{i}; \IR))   + L(n)N_{c}(M_i).
\end{align}
for some positive constant $L(n)>0$ \emph{independent} of the index $i$.  This follows from the fact that every cusp cross section is flat, and that by Bieberbach's theorem there are only finitely many diffeomorphism types of flat manifolds with given dimension.  Suppose now $n=2m+1$. By Theorem \ref{Zucker} we have: 
\[
\dim_{\IR}Im\big(H^{m}_c(M_i; \IR)\to H^{m}(M_{i}; \IR)) =\dim_{\IR}\mathcal{H}^m_{g_{\IR}}(M_i),
\]
and so we obtain:
\begin{align}\label{Hkbound}
\frac{\dim_{\IR}H^{k}(M_{i}; \IR)}{Vol(M_{i})}\leq L(n)R_{g_{\IR}}(M_{i})+\frac{b^{k}_{2}(M_i)}{Vol (M_{i})}.
\end{align}
We are now ready to prove the following proposition.

\begin{proposition}
Let $(M^{n}:=\Gamma\backslash\textbf{H}^{n}_{\IR}, g_{\IR})$ be a complete finite volume real-hyperbolic manifold of dimension $n=2k+1$, for some positive integer $k$. Given a cofinal filtration $\{\Gamma_{i}\}$ of $\Gamma$, let $\pi_{i}: M_{i}\rightarrow M$ denote the regular Riemannian cover of $M$ associated to $\Gamma_{i}$. Set $g_{i}:=\pi^*_ig_{\IR}$. Decompose $M^n$ as a disjoint union
\[
M^n=M_{0}\cup(\cup_j C_j),
\]
where $M_0$ is a compact manifold with boundary and each $C_j$ is a cusp.  Define $M^i_{0}:=\pi^{-1}_{i}(M_{0})$, and let $\inj_{M^i_{0}}:=\inf_{x\in M^i_{0}} inj_{g_{i}}(x)$. There exists a constant $\lambda(k)$ such that
\begin{align}\label{logdecay}
\frac{\dim_{\IR}H^{k}(M_{i}; \IR)}{Vol(M_{i})}\leq \frac{\lambda(k)}{\ln(V_{min}(M^{i}_{0}))}. 
\end{align}
In particular, the sub volume growth of $\dim_{\IR}H^{k}(M_{i}; \IR)$ is at least of the order $\inj^{-1}_{M^i_{0}}$.
\end{proposition}

\begin{proof}
Combining Equation \eqref{Hkbound}, Lemma \ref{normalizedcusps} and Theorem \ref{realhypcusps} we obtain
\[
\frac{\dim_{\IR}H^k(M_{i}; \IR)}{Vol (M_{i})}\leq L(n)Vol(B_{\frac{1}{2}\inj_{M_0^i}})^{-1} +\frac{c_{3}(n, k)}{\inj_{M^i_{0}}},
\]
so that we can find another constant $\lambda(k)>0$ so that Equation \eqref{logdecay} is satisfied. In particular, by Theorem \ref{fat} we have
\[
\lim_{i\to\infty}\frac{\dim_{\IR}H^{k}(M_{i}; \IR)}{Vol(M_{i})}=0,
\]
with a decay of the order at least $\inj^{-1}_{M^i_{0}}$.
\end{proof}

\section{Appendix}
\subsection{Comparison with the Trace Formula}
In this appendix, we sharpen standard matrix coefficient estimates to show how such techniques can be used (albeit at the expense of much greater computation) to derive an alternate proof of a subset of our results in the earlier sections of this paper.

Let $G $ be a semisimple Lie group, $K\subset G$ a maximal compact subgroup, and $\Gamma\subset G$ a discrete torsion free cocompact  subgroup. Let $\mathfrak{g}$ and $\mathfrak{k}$ denote the lie algebras of $G$ and $K$ respectively. Write the Cartan decomposition of $\mathfrak{g}$ as 
$\mathfrak{g} = \mathfrak{k}\oplus \mathfrak{p}$. Let $M:=\Gamma\backslash G/K$ denote the corresponding compact locally symmetric space. In \cite{Matsushima}, Matsushima established the following link between cohomology and representation theory.
\begin{theorem}[Matsushima] 
For an irreducible unitary representation $\omega$ of $G$, let $\omega_K$ denote the restriction of $\omega$ to $K$, $N(\Gamma, \omega)$ denote the multiplicity of $\omega$  in $L^2(\Gamma\backslash G)$,  and $m_k(\omega_K)$ denote the multiplicity of $\omega_K$ in $\Lambda^k\mathfrak{p}^*$. 
Then the Betti numbers are given by the following finite sum: 
\begin{align}\label{Matsushima}b^k(M) = \sum_{\omega:\omega\text{ has vanishing casimir}}m_k(\omega_K)N(\Gamma,\omega).
\end{align}
\end{theorem} 

The Selberg trace formula provides a mechanism for expressing the $N(\Gamma,\omega)$ in terms orbital integrals, which may, however, be difficult to interpret. In \cite{DW78} DeGeorge-Wallach  introduced a method for estimating these multiplicities from above, which we now recall.

 Let $\chi$ denote the characteristic function of the identity component of $\pi^{-1}(B_{\inj_M}(\Gamma K))\subset G$, where $\pi:G\to M$ is the natural projection.  Let $\omega$ be an irreducible unitary  representation of $G$ on a Hilbert space  $H_\omega$. Let $v\in H_\omega$ be a unit vector. Let $N(\Gamma,\omega)$
denote the multiplicity of $\omega$ in $L^2(\Gamma\backslash G)$. Then DeGeorge-Wallach use the trace formula to show that \cite[Corollary 3.2]{DW78}
\begin{align}\label{DW}
\frac{N(\Gamma, \omega)}{Vol(\Gamma\backslash G)}\leq \frac{1}{\int_G\chi(g)\langle \omega(g)v,v\rangle^2dv}.
\end{align}
The role of $\chi$ is to ensure that only the (conjugacy class of the) identity enters the trace formula.
Hence lower bounds on the absolute value of the matrix coefficient $\langle \omega(g)v,v\rangle$ become upper bounds on the normalized multiplicity $\frac{N(\Gamma,\omega)}{Vol(\Gamma\backslash G)}$.

For the geometer, perhaps the most interesting instance of $H_\omega$ is the subspace of $L^2(\Gamma\backslash G)$ generated by $L^2$-harmonic forms of suitable $K$-invariance type. Let $h$ be a harmonic $m$-form on $M$  with $L^2$-norm one. Let $\pi:\Gamma\backslash G\to M$ be the quotient map.  Using left invariant vector fields to define a canonical trivialization of $\Lambda^mT^*(\Gamma\backslash G)$, the metric induces a pairing for all $x,y\in \Gamma\backslash G$, 
$$\langle\cdot,\cdot\rangle:\Lambda^mT^*_{x}(\Gamma\backslash G)\times \Lambda^mT^*_y(\Gamma\backslash G)\to \IR. $$
Let $\sigma$ denote the adjoint representation of $K$ on $\Lambda^m\mathfrak{p}^*$. Viewed as a $\Lambda^m\mathfrak{p}^*$ valued function, $\pi^*h$ satisfies for all $k\in K$
\begin{align}
\label{ktrans}\pi^*h(gk)=\sigma^{-1}(k)h(g).
\end{align} Choosing $v\in H_\omega$ to be $\pi^*h$, the matrix coefficient becomes 
\begin{align}\label{matcoeff}\langle \omega(g)v,v\rangle = \int_{\Gamma\backslash G}\langle \pi^*h (xg),\pi^*h(x)\rangle dv_x,\end{align}
and lower bounds on this inner product lead to upper bounds on $b^k(M)$ via \eqref{Matsushima}.  

 The estimate given in \cite{DS17}, on the other hand, essentially reduces to 
\[
\frac{b^m(M)}{Vol(M)}\leq Csup_{p\in M, h\not = 0}\frac{\int_{B_1(p)}|h|^2dv}{\int_{B_{\inj_M}(p)}|h|^2dv},
\]
for some constant $C$ depending only on the curvature of $M$, if we assume that $\inj_M>1$. Hence lower bounds on the growth of norms of harmonic forms is once again the key component of the estimate. In this more geometric approach, the restriction to embedded balls in $M$ rather than larger balls in the universal cover is not a function of the Selberg trace formula, but is instead a simple mechanism to ensure that 
$\int_{B_{R}(\tilde p)\subset \tilde M}|\omega^*h|^2dv\leq \int_M|h|^2dv$,  where $\omega:\tilde M \to M$ is the  universal cover.  The role of the Price inequality  - like the differential equations governing the matrix coefficients - is to provide a mechanism for estimating the growth of $\int_{B_{r}(p)}|h|^2dv$ as a function of $r$.

In the remainder of this  appendix we present a direct analysis of matrix coefficients arising in the estimation of the first Betti numbers of compact quotients of complex-hyperbolic spaces and all but middle degree Betti numbers of compact quotients of real-hyperbolic spaces. The same method, albeit with additional complications, can be used to estimate the matrix coefficients required to bound all the Betti numbers outside middle degree for compact complex-hyperbolic spaces, but we do not include those computations here, as they are significantly longer than the method used in Section  \ref{Complex Hyperbolic} and simply reproduce the results given there.  There is, of course, a vast literature on the estimation of matrix coefficients, but we found our approach in these cases so much simpler than the general theory, that we felt it might be useful, especially to the non-expert, to include it here.  

 
\subsection{Real rank one generalities}\label{generalities}
Let $G$ be a real rank one semisimple Lie group. 
Fix a maximal $\IR$-split torus $A\subset G$, and a Weyl chamber $\mathfrak{a}^+$. Denote the positive restricted roots of $A$ by  $\lambda$ and $ 2\lambda $. Let $m(\lambda)$ and $m(2\lambda)$ denote their multiplicities.  Choose an element $u\in \mathfrak{a}^+$ so that  $\lambda(u)=1.$   Let $U$ denote the covector metrically dual to $u$. Write $a(t):= exp(tu).$

Using the $KAK$ decomposition of $G$, we reduce the analysis of the matrix coefficient $\eqref{matcoeff}$ to the study of the function  
\begin{align}\phi(t):= \int_{\Gamma\backslash G}\langle \pi^*h(xa(t)),\pi^*h(x)\rangle dv_x.\end{align}
In particular, it is easy to show that for some $c(G,m)>0$ depending only on $G$ and $m$, with $\langle \omega(g)v,v\rangle$ as in \eqref{matcoeff}, 
\begin{align}\label{above}\int_{B_R}\langle \omega(g)v,v\rangle^2dg \geq c(G,m)\int_0^R\phi(t)^2\sinh(t)^{m(\lambda)+m(2\lambda)}\cosh(t)^{m(2\lambda)} dt.\end{align}

Combining \eqref{Matsushima}, \eqref{DW}, and \eqref{above}, we obtain for some $C(G,m)>0$, depending only on $G$ and $m$, 
\begin{align}\label{DWA}\frac{b^m(\Gamma\backslash G/K)}{Vol(\Gamma\backslash G/K)}\leq \frac{C(G,m)}{\int_0^{\inj_M}\phi(t)^2\sinh(t)^{m(\lambda)+m(2\lambda)}\cosh(t)^{m(2\lambda)} dt}.
\end{align}
We prove the following results.

\begin{proposition}\label{propa}
Let $G=SU(n,1)$ and $m=1$. 
For $t>1,$ there are constants $c_1(n),c_2(n)>0$ such that 
\begin{align}
\frac{c_1(n)}{\sinh(t)}\leq \phi(t)\leq \frac{c_2(n)}{\sinh(t)}.
\end{align}
\end{proposition}
\begin{proposition}\label{propb}
Let $G=SO(n,1)$ and $m<\frac{n-1}{2}$. 
For $t>1,$ there is a constant $c(n,m) >0$ such that 
\begin{align}
\frac{c(n,m)}{\sinh^m(t)}\leq \phi(t) .
\end{align}
For $m=\frac{n-1}{2}$, and $t>1,$ there is a constant $c_0(n) >0$ such that 
\begin{align}
\frac{c_0(n)}{t}\leq \phi(t) .
\end{align}
\end{proposition}
As a corollary we recover (with different constants) our prior sections' Betti number bounds for compact real- and complex-hyperbolic manifolds (but only in degree  $1$ in the complex-hyperbolic case here).

\subsection{$G=SU(n,1)$}
In this section we prove Proposition \ref{propa}.\\

In order to estimate $\phi$, we must also understand 
$$\psi(t):=  \int_{\Gamma\backslash G}\langle (e(U)e^*(U)+e(JU)e^*(JU))\pi^*h(xa(t)),\pi^*h(x)\rangle dv_x.$$
$\phi$ and $\psi$ satisfy differential equations following from the fact that $0 = dh = d^*h = dJh=d^*Jh$. Choose an orthonormal basis of $\mathfrak{p}$ of the form $\{X_j\}_{j=1}^{2n-2}\cup\{u,Ju\},$ where the $X_j$ are in the span of the $\pm\lambda$ root spaces. $Ju$ is in the span of the $\pm 2\lambda$ root spaces. Let $\{\omega^j\}_{j=1}^{2n-2}\cup\{U,JU\} $ denote the dual frame. With respect to the framing by left invariant vector fields
\begin{align}d\pi^*h = \left(e(\omega^j)X_j+ e(U)u+e(JU)Ju\right)\pi^*h,
\end{align}
and
\begin{align}d^*\pi^*h = -\left(e^*(\omega^j)X_j+ e^*(U)u+e^*(JU)Ju\right)\pi^*h.
\end{align}
For a proof of the these identities, see for example \cite[Equations (4.12) and (6.3)]{MM}. Henceforth, we will simply write $h$ instead of $\pi^*h$. 
Hence we have 
\begin{align}\label{A1}
0 &= \int_{\Gamma\backslash G}\langle e^*(U)(e(\omega^j)X_j+ e(U)u+e(JU)Ju)h(xa(t)), h(x)\rangle dv_x\nonumber\\
&= \frac{d}{dt}\int_{\Gamma\backslash G}\langle e^*(U) e(U) h(xa(t)), h(x)\rangle dv_x+\nonumber \\
& \quad \int_{\Gamma\backslash G}\langle e^*(U)(e(\omega^j)X_j +e(JU)Ju)h(xa(t)), h(x)\rangle dv_x,
\end{align}
and 
\begin{align}\label{A2}0& = \int_{\Gamma\backslash G}\langle e(U)(e^*(\omega^j)X_j+ e^*(U)u+e^*(JU)Ju)h(xa(t)), h(x)\rangle dv_x\nonumber\\
&=\frac{d}{dt}\int_{\Gamma\backslash G}\langle e(U) e^*(U) h(xa(t)), h(x)\rangle dv_x+\nonumber \\
&\quad\int_{\Gamma\backslash G}\langle e(U)(e^*(\omega^j)X_j +e^*(JU)Ju)h(xa(t)), h(x)\rangle dv_x.
\end{align}
We have two additional equations obtained by replacing $h$ by $Jh$. 
To generate differential equations for $\phi$ and $\psi$ from \eqref{A1} and \eqref{A2}, we follow the first step of the strategy indicated in \cite[Chapter VIII Section 7]{Knapp} and illustrated in a special case in \cite[Lemma 8.15]{Knapp}. 
Write 
\begin{align}\label{knapp}
X_j = \coth(t)Y_j - \frac{1}{\sinh(t)}a^{-1}(t)Y_ja(t),\quad Ju = \coth(t)Y_{2\lambda} - \frac{1}{\sinh(t)}a^{-1}(t)Y_{2\lambda}a(t).
\end{align} 
Here $Y_j$ and $Y_{2\lambda}\in \mathfrak{k}$ are defined as follows. Write $X_j = N_j+N_j^*,$ where $N_j$ is in the $\lambda$ root space. Then $Y_j = N_j-N_j^*.$ $Y_{2\lambda}$ is similarly defined. Combining \eqref{knapp} and \eqref{ktrans}, we have 
\begin{align}\label{asub}&\int_{\Gamma\backslash G}\langle e^*(U)e(\omega^j)X_j h(xa(t)), h(x)\rangle dv_x\nonumber\\
&= \coth(t)\int_{\Gamma\backslash G}\langle e^*(U) e(\omega^j)Y_j h(xa(t)), h(x)\rangle dv_x\nonumber\\
&-\frac{1}{\sinh(t)}\int_{\Gamma\backslash G}\langle e^*(U) e(\omega^j) a^{-1}(t)Y_j a(t)h(xa(t)), h(x)\rangle dv_x\nonumber\\
&= -\coth(t)\int_{\Gamma\backslash G}\langle e^*(U) e(\omega^j)\sigma(Y_j) h(xa(t)), h(x)\rangle dv_x\nonumber\\
&-\frac{1}{\sinh(t)}\int_{\Gamma\backslash G}\langle e^*(U) e(\omega^j) \frac{d}{ds}_{|s=0}h(x\exp(sY_j)a(t)), h(x)\rangle dv_x\nonumber\\
&= -\coth(t)\int_{\Gamma\backslash G}\langle e^*(U) e(\omega^j)\sigma(Y_j) h(xa(t)), h(x)\rangle dv_x\nonumber\\
&-\frac{1}{\sinh(t)}\int_{\Gamma\backslash G}\langle e^*(U)(e(\omega^j) \frac{d}{ds}_{|s=0}h(xa(t)), h(x\exp(-sY_j))\rangle dv_x\nonumber\\
&= -\coth(t)\int_{\Gamma\backslash G}\langle e^*(U) e(\omega^j)\sigma(Y_j) h(xa(t)), h(x)\rangle dv_x\nonumber\\
&-\frac{1}{\sinh(t)}\int_{\Gamma\backslash G}\langle \sigma(Y_j)e^*(U)(e(\omega^j)  h(xa(t)), h(x )\rangle dv_x.
\end{align}
We may similarly replace the $Ju$ derivatives with $-\coth(2t)\sigma(Y_{2\lambda})$ and $\frac{1}{\sinh(2t)}\sigma(Y_{2\lambda})$ terms. 
The adjoint representation satisfies 
\begin{align}\label{Yj1} \sigma(Y_{j}) = e(u)e^*(X_{j})-e(X_{j})e^*(u)+ e(JX_{j})e^*(Ju)-e(Ju)e^*(JX_{j}).\end{align}
\begin{align}\label{Y21} \sigma(Y_{2\lambda}) = 2e(u)e^*(Ju)-2e(Ju)e^*(u)  -e(JX_{j})e^*(X_{j}).\end{align}

Taking the closed and coclosed equations \eqref{A1} and \eqref{A2} and the corresponding equations with $h$ replaced by $Jh$,  replacing $X_j$ and $Ju$ by combinations of $\sigma(Y_{j})$ and $\sigma(Y_{2\lambda})$ as indicated in \eqref{asub}, and substituting in expressions \eqref{Yj1} and \eqref{Y21} for $\sigma(Y_{j})$ and $\sigma(Y_{2\lambda})$, yields a system of 4 differential equations, from which we obtain 
\begin{align}\label{d1f}0 &=  (\phi-\psi)'+\Big(\coth(t)+\frac{1}{\sinh(t)}\Big)(\phi -\psi) -(n-1)\Big[ \coth(t)   + \frac{1}{\sinh(t)} \Big]\psi(t).
\end{align}
and
\begin{align}\label{ds1s}0 &= \psi'(t) + [(2n-2)\coth(t)+  2\tanh(t)]\psi(t)
 -\frac{2}{\sinh(t)}(\phi(t)-\psi(t)),
\end{align}
with initial conditions $\phi(0) = 1$, and $0<\psi(0)<1$. 
From these equations, we see that $\phi(t)-\psi(t)$ cannot become zero while $\psi(t)$ is positive and $\psi$ cannot vanish while $\phi-\psi$ is positive. Hence $\phi(t)>\psi(t) >0$ for all $t$.
  
Next, we combine the equations and introduce integrating factors to write 
\begin{align}\label{d1fn} 
& \Big[\sinh(t)\tanh\Big(\frac{t}{2}\Big)^{2n-1}(\phi-n\psi)\Big]^\prime\nonumber\\
& =(2n-2)\Big[(n-1)\frac{\cosh(t)-1}{\sinh(t)}+   \tanh(t)\Big]\sinh(t)\tanh\Big(\frac{t}{2}\Big)^{2n-1}\psi  .
\end{align}
Hence for any $t$, we have $\phi(t)\geq n\psi(t)$. We also have the expression
\begin{align}\label{d1s}\nonumber 
\Big(\frac{\sinh(t)}{\tanh(\frac{t}{2})}\phi\Big)^\prime &=-
\Big[(n-2)\frac{\cosh(t)-1}{\sinh(t)}+ 2\tanh(t) \Big]\frac{\sinh(t)}{\tanh(\frac{t}{2})}\psi(t)\nonumber\\
&\geq -\Big[(n-2)\frac{\cosh(t)-1}{\sinh(t)}+ 2\tanh(t) \Big]\frac{1}{n}\frac{\sinh(t)}{\tanh(\frac{t}{2})}\phi(t) 
\nonumber\\
&\geq -C_n\frac{\sinh(t)}{\tanh(\frac{t}{2})}\phi(t),
 \end{align}
where $C_n:= \sup_t \frac{1}{n}\big[(n-2)\frac{\cosh(t)-1}{\sinh(t)}+ 2\tanh(t)\big].$ Hence 
$\frac{\sinh(t)}{\tanh(\frac{t}{2})}\phi(t)\geq e^{-tC_n},$ and $\phi(t)\leq \frac{\tanh(\frac{t}{2})}{\sinh(t)}.$ 
Finally from the expression 
\begin{align}\label{d1ff}
\Big[\sinh(t)\tanh\Big(\frac{t}{2}\Big)(\phi-\psi)\Big]^\prime =(n-1)[ \cosh(t)   + 1 ] \tanh\Big(\frac{t}{2}\Big)\psi(t),
\end{align}
we see that
\begin{align}\nonumber  
\sinh(t)\tanh\Big(\frac{t}{2}\Big)(\phi-\psi)\geq \sinh(1)\tanh\Big(\frac{1}{2}\Big)(\phi(1)-\psi(1))
\geq \frac{n-1}{n} \tanh^2\Big(\frac{1}{2}\Big) e^{-C_n}.
\end{align}
In particular, for $t>1$, 
\begin{align} \nonumber
\frac{\tanh(\frac{t}{2})}{\sinh(t)}\leq \phi(t)\leq \frac{c(n)}{\sinh(t)},
\end{align}
proving Proposition \ref{propa}.\\
 
\subsection{$G=SO(n,1)$}

In this subsection, we let $h$ be an $L^2$-norm one harmonic $k$-form on a compact real-hyperbolic $n$-manifold $M = \Gamma\backslash SO(n,1)/SO(n)$. Let $\pi$ denote the projection map $\pi: \Gamma\backslash SO(n,1)\to M$. Define the matrix coefficient 
$$f(t):= \int_{\Gamma\backslash SO(n,1)}\langle \pi^*h(xa(t)),\pi^*h(x)\rangle dv_x,$$
with $a(t):=\exp(tu)$, $u$ defined as in the preceding subsection. Define also the auxiliary function 
$$b(t):= \int_{\Gamma\backslash SO(n,1)}\langle e(u)e^*(u)\pi^*h(x a(t)),\pi^*h(x)\rangle dv_x.$$
Computing as in the previous subsection, one derives the following coupled equations for  $f$ and $b$: 
\begin{align}\label{a}
(f-b)'+k\coth(t)(f-b)-\frac{n-k}{\sinh(t)}b=0,
\end{align}

\begin{align}\label{b}
b'+(n-k)\coth(t)b-\frac{k}{\sinh(t)}(f-b)=0,
\end{align}
with initial conditions $f(0) = 1$, $0< b(0)<1$. Introduce integrating factors to rewrite these equations as

\begin{align}\nonumber
\frac{d}{dt}(\sinh^k(t)(f-b)) =(n-k)\sinh^{k-1}(t)b,
\end{align}
\begin{align}\nonumber
\frac{d}{dt}(\sinh^{n-k}(t)b)= k\sinh^{n-k-1}(t)(f-b).
\end{align}
From these equations, we see that $f>b>0$ for all $t$, and for $t\geq 1$, 
\begin{align}\label{basic}
f(t)\geq \frac{\sinh^k(1)(f(1)-b(1))}{\sinh^k(t)}.
\end{align}
Next, rewrite Equation \ref{a} as 
\begin{align}\nonumber
0 &= \Big(f-\frac{n}{k}b+\frac{n-k}{k}b\Big)'+k\coth(t)\Big(f-\frac{n}{k}b+\frac{n-k}{k}b\Big)-\frac{n-k}{\sinh(t)}b \nonumber\\
&= \Big(f-\frac{n}{k}b\Big)'+\frac{n-k}{k}b'+k\coth(t)\Big(f-\frac{n}{k}b\Big)+(n-k)\coth(t)b -\frac{n-k}{\sinh(t)}b, \nonumber
\end{align}
so that by substituting the identity for $b^\prime$ given in \eqref{b} we obtain
\begin{align}\nonumber
\Big(f-\frac{n}{k}b\Big)^\prime + \frac{n-k+k\cosh(t)}{\sinh(t)}\Big(f-\frac{n}{k}b\Big)= \frac{(n-k)(n-2k)( \cosh(t)-1)}{k\sinh(t)}b>0.
\end{align}
This also shows that 
\begin{align}\label{bf}
b\leq\frac{k}{n}f.
\end{align} 
Now, by summing Equations \eqref{a} and \eqref{b} we obtain
\begin{align}\nonumber
f^\prime=-k\coth(t)f+\frac{n-2k}{\sinh(t)}b-(n-2k)\coth(t)b+\frac{k}{\sinh(t)}f.
\end{align}
By using this expression for $f^\prime$, and after some manipulations with hyperbolic functions one finds
\begin{align}\nonumber
\Big(\frac{\sinh^k(t)}{\tanh^k(\frac{t}{2})}f\Big)^\prime +\frac{\sinh^k(t)}{\tanh^k(\frac{t}{2})}\frac{(n-2k)(\cosh(t)-1)}{\sinh(t)}b=0,
\end{align}
and then 
\begin{align}\label{prec} 
\Big(\frac{\sinh^k(t)}{\tanh^k(\frac{t}{2})}f\Big)'+\frac{2(n-2k)\cosh^2(\frac{t}{2})\sinh^{k-1}(t)}{\tanh^{k-2}(\frac{t}{2})}b=0.
\end{align}
Next, for $2k<n$, we use \eqref{bf} to replace \eqref{prec} by a differential inequality 
\begin{align}\nonumber
\Big(\frac{\sinh^k(t)}{\tanh^k(\frac{t}{2})}f\Big)' \geq-\frac{2(n-2k)\cosh^2(\frac{t}{2})\sinh^{k-1}(t)}{\tanh^{k-2}(\frac{t}{2})}\frac{k}{n}f.
\end{align}
Since 
\[
\lim_{t\to 0}\frac{\sinh^k(t)}{\tanh^k(\frac{t}{2})}=2^k,
\]
integration from $0$ to $t$ yields 
\begin{align}\nonumber
f(t) \geq 2^k\frac{\tanh^k(\frac{t}{2})}{\sinh^k(t)}e^{-\int_0^t\frac{2(n-2k)\cosh^2(\frac{s}{2})\tanh^{2}(\frac{s}{2})}{\sinh(s)}\frac{k}{n}ds}.
\end{align}
In particular, we have
\[
f(1)\geq 2^k\frac{\tanh^k(\frac{1}{2})}{\sinh^k(1)}e^{-\int_0^1\frac{2(n-2k)\cosh^2(\frac{s}{2})\tanh^{2}(\frac{s}{2})}{\sinh(s)}\frac{k}{n}ds},
\]
so that by evaluating \eqref{bf} at $t=1$ and using the estimate in \eqref{basic} then yields
\begin{align}\nonumber
f(t)\geq \frac{c}{\sinh^k(t)} ,
\end{align}
with 
\[
c= e^{-\int_0^1\frac{2(n-2k)\cosh^2(\frac{s}{2})\tanh^{2}(\frac{s}{2})}{\sinh(s)}\frac{k}{n}ds}\frac{(n-k)2^k \tanh^k(\frac{1}{2})}{n}.
\]
Hence, for some constant $c(n,k)>0$, we have 
\begin{equation}\label{finalestimate} 
\int_0^Rf(t)^2\sinh^{n-1}(t)dt\geq  
\left\{ \begin{array}{rl} &c(n,k)e^{(n-1-2k)R}, \quad \text{if}\quad n-1>2k;\\
&\\
&c(n,k)R,\quad  \text{if} \quad n-1=2k.\\
\end{array} \right.
\end{equation}
In the notation used in Section \ref{generalities} we have $\phi=f$, so that a proof of Proposition \ref{propb} follows from the estimate in Equation \ref{finalestimate}. \\ \\ \\ \\

\end{document}